\documentclass[11pt]{article}
\usepackage{amsfonts,amsmath,amssymb}
\usepackage{hyperref,amsthm}
\usepackage{epsfig}
\usepackage{graphicx}

\newtheorem{theorem}{Theorem}[section]
\newtheorem{proposition}[theorem]{Proposition}
\newtheorem{lemma}[theorem]{Lemma}
\newtheorem{claim}[theorem]{Claim}
\newtheorem{corollary}[theorem]{Corollary}
\newtheorem{remark}[theorem]{Remark}
\newtheorem{notation}[theorem]{Notation}
\newtheorem{question}[theorem]{Question}

\newtheorem{definition}[theorem]{Definition}

\begin{document}

\title{Reconstruction of the Geometric Structure of a Set of Points in the Plane
from its Geometric Tree Graph}

\author{
Chaya Keller and Micha A. Perles \\
Einstein Institute of Mathematics, Hebrew University\\
Jerusalem 91904, Israel\\
}

\maketitle

\begin{abstract}
Let $P$ be a finite set of points in general position in the plane.
The structure of the complete graph $K(P)$ as a geometric graph includes,
for any pair $[a,b],[c,d]$ of vertex-disjoint edges, the information whether
they cross or not.

The simple (i.e., non-crossing) spanning trees (SSTs) of $K(P)$ are the vertices
of the so-called Geometric Tree Graph of $P$, $G(P)$. Two such vertices are
adjacent in $G(P)$ if they differ in exactly two edges, i.e., if one can be
obtained from the other by deleting an edge and adding another edge.

In this paper we show how to reconstruct from $G(P)$ (regarded as an abstract graph)
the structure of $K(P)$ as a geometric graph. We first identify within $G(P)$
the vertices that correspond to spanning stars. Then we regard each star $S(z)$
with center $z$ as the representative in $G(P)$ of the vertex $z$ of $K(P)$.
(This correspondence is determined only up to an automorphism of $K(P)$ as a
geometric graph.) Finally we determine for any four distinct stars $S(a), S(b),
S(c),$ and $S(d)$, by looking at their relative positions in $G(P)$, whether the
corresponding segments cross.
\end{abstract}

\section{Introduction}

Graph reconstruction is an old and extensive research topic. It dates back to the
{\it Reconstruction Conjecture} raised by Kelly and Ulam in 1941 (see~\cite{Kelly,Ulam}),
which asserts that every graph on at least three vertices is uniquely determined by its collection of
vertex deleted subgraphs.

As a natural extension of the Reconstruction Conjecture, numerous papers considered
either reconstruction of structures other then graphs (a research topic proposed
by Ulam in 1960), or reconstructions of graphs from other information. In the first direction,
reconstructed objects include colored graphs, hypergraphs, matroids, relations, and other
classes. In the second direction, the ``information'' may be $k$-vertex deleted subgraphs,
edge-deleted subgraphs, elementary contractions, spanning trees, etc. In addition,
various papers considered reconstruction of {\it parameters of the graph} instead of its full
structure. Such parameters include the order, the degree sequence, planarity, the types of
spanning trees, and many others (see the surveys~\cite{Bondy,Rama} for references).

In this paper, we study the problem of reconstructing the geometric structure of a
set of points in the plane from its geometric tree graph.

\emph{Tree graphs} were defined in 1966 by Cummins~\cite{Cummins} in the context of listing
all spanning trees of a given connected graph effectively. The {\it tree graph} $T(G)$ of a graph
$G$ has the spanning trees of $G$ as its vertices, and two spanning trees are adjacent if
one can be obtained from the other by deleting an edge and adding another edge. These graphs
were studied in a number of papers and were shown to be Hamiltonian and to have the maximal
possible connectivity (see, e.g.,~\cite{HH72,Liu88}).

In 1996, Avis and Fukuda~\cite{AvisFukuda} defined the \emph{geometric tree graph}, as the
counterpart of tree graphs in the geometric graph setting.
\begin{definition}
Let $P$ be a finite point set in general position in the plane. The \emph{geometric tree graph} $G(P)$ is defined
as follows. The vertices of $G(P)$ are the simple (i.e., non-crossing) spanning trees (SSTs)
of $K(P)$. Two such vertices are adjacent in $G(P)$ if they differ in exactly two edges, i.e.,
if one can be obtained from the other by deleting an edge and adding another edge.
\end{definition}
Geometric tree graphs were shown to be connected~\cite{AvisFukuda}, and upper and lower
bounds on their diameter were established~\cite{AvisFukuda,Hernando1}.

\medskip

We study a reconstruction problem for geometric graphs: Is the geometric tree graph $G(P)$
sufficient for ``reconstructing'' the structure of $K(P)$? In a sense, this question is a geometric
counterpart of the work of Sedl\'{a}\v{c}ek~\cite{Sedlacek}, who studied the question whether a
graph can be reconstructed from its spanning trees. As we deal with a geometric setting, we seek
to reconstruct the \emph{geometric structure} of the graph.
\begin{definition}
Let $P$ be a finite set of points in general position in the plane. The \emph{geometric structure}
of the complete graph $K(P)$ as a geometric graph includes, for any pair $[a,b],[c,d]$ of
vertex-disjoint edges, the information whether they cross or not.
\end{definition}
Our main result is the following:
\begin{theorem}\label{Thm:Main}
For any finite set $P$ of points in general position in the plane, the geometric structure
of $K(P)$ can be reconstructed from the geometric tree graph $G(P)$.
\end{theorem}
While the proof of the theorem is elementary, it is rather complex, and consists of several
stages:
\begin{enumerate}
\item \textbf{Maximal cliques in $G(P)$.} We study thoroughly the structure of maximal cliques
in $G(P)$. We divide these cliques into two types, called ``union max-cliques'' and
``intersection max-cliques'', and show that given a maximal clique in $G(P)$, one can determine
its type. This study spans Section~\ref{sec:max-cliques}.

\item \textbf{Stars and brushes in $G(P)$.} We show how to identify the vertices of $G(P)$
that correspond to spanning stars and spanning brushes (i.e., spanning trees of diameter 3 with
a single internal edge), by examining the max-cliques to which they belong. The stars
are determined only up to an automorphism of $K(P)$ (obviously, one cannot do better), and once
they are fixed, the brushes are determined uniquely. This part of the proof is presented in
Section~\ref{sec:stars-and-brushes}.

\item \textbf{The geometric structure of $K(P)$.} We show how the geometric structure of $K(P)$ can
be derived from information on the brushes in $G(P)$. This part is presented in
Section~\ref{sec:geom-structure}.
\end{enumerate}

In the last part of the paper, Section~\ref{sec:general}, we consider abstract (i.e.,
non-geometric) graphs, and show that a variant of the argument developed in
Sections~\ref{sec:max-cliques} and~\ref{sec:stars-and-brushes}
can be used to prove the following result:
\begin{theorem}\label{Thm:Main-General}
For any $n \in \mathbb{N}$, the automorphism group of the tree graph of $K_n$ is 
isomorphic to $\mathrm{Aut}(K_n) \cong S_n$.
\end{theorem}

Our treatment of the geometric reconstruction problem (i.e., $K(P)$ from $G(P)$) falls
short of this. It leaves open the (quite implausible) possibility that the geometric tree 
graph $G(P)$ has an automorphism $\eta$, other than the identity, that fixes each star
and each brush. This leaves open, for further research, the following question.
\begin{question}
Is this true that for any finite set $P$ of points in general position in the plane, 
we have $\mathrm{Aut}(G(P)) \cong \mathrm{Aut}(K(P))$, where $G(P)$ is treated as an 
abstract graph, whereas $K(P)$ is treated as a geometric graph?
\end{question}

\section{Maximal Cliques in $G(P)$}
\label{sec:max-cliques}

In this section we study the structure of maximal (with respect to inclusion) cliques in the
geometric tree graph $G(P)$. We divide the maximal cliques into two types, called U-cliques
and I-cliques, and our ultimate goal is to determine, given a maximal clique in $G(P)$, what
is its type.

We start in Section~\ref{sec:sub:notations} with a few definitions and notations, to be used
throughout the paper. In Sections~\ref{sec:sub:max-cliques-basics} and~\ref{sec:sub:geom-meaning}
we describe a classification of the maximal cliques into two types, presented originally in~\cite{Virginia-PhD},
and discuss basic properties of both types. In order to distinguish between general combinatorial
considerations and geometric arguments specific to SSTs, we start in
Section~\ref{sec:sub:max-cliques-basics} with a general combinatorial
framework, and leave the geometric arguments to Section~\ref{sec:sub:geom-meaning}.

In Sections~\ref{sec:sub:char-degenerate} and~\ref{sec:sub:identify-degenerate} we study
{\it degenerate} maximal cliques, i.e., maximal cliques of size 2. In
Section~\ref{sec:sub:char-degenerate} we give a geometric characterization of the situation
when a maximal clique is degenerate, and in Section~\ref{sec:sub:identify-degenerate}
we show how to identify whether a given degenerate maximal clique is a U-clique or an
I-clique. Finally, in Section~\ref{sec:sub:identify-non-degenerate} we show how to determine 
whether a given {\it non-degenerate} maximal clique is a U-clique or an I-clique. 

\subsection{Definitions and Notations}
\label{sec:sub:notations}

\begin{notation}
The following notations and conventions are used throughout the paper. The straight
line that passes through points $x,y$ is denoted by $\ell(x,y)$. The vertex and
edge sets of a graph $G$ are denoted by $V(G)$ and $E(G)$, respectively.
Since the set $P$ is fixed, we shall always identify a spanning subgraph of $K(P)$
(and, in particular, a spanning subtree) with its set of edges.
\end{notation}

In our study we shall extensively use {\it maximal cliques} of $G(P)$. These are defined as follows:
\begin{definition}
A \emph{max-clique} in a graph $G$ is a maximal (with respect to inclusion) clique
included in $G$. Since any max-clique is a complete graph on its vertex set, we shall
identify a max-clique with its set of vertices.
\end{definition}
We shall use the following observation on the structure of $G(P)$, proved by Avis and
Fukuda~\cite{AvisFukuda}.
\begin{claim}[~\cite{AvisFukuda}, Lemma 3.15]\label{Claim:G(P)-connected}
For any set $P$ of points in general position in the plane, $G(P)$ is connected and
its diameter is $\leq 2|P|-4$.
\end{claim}

\subsection{Types of Max-cliques in Tree Graphs}\label{sec:sub:max-cliques-basics}

A generic combinatorial way to treat a tree graph is to consider a base set $X$
and a graph $G$ whose vertices are $q$-subsets of $X$ (not necessarily all the $q$-subsets),
such that two vertices $A, B \in V(G)$ are adjacent if and only if $|A \triangle B|=2$.
(In the case of the (geometric) tree graph of a point set $P$, we have $X=E(K(P))$,
$q=|P|-1$, and the vertices of $G$ are the sets of edges of (simple) spanning trees of $K(P)$.)

Let $A, B$ be two adjacent vertices of $G$. Denote
\[
I=A \cap B, \qquad D = A \triangle B \qquad \mbox{ and } \qquad U = A \cup B = I \cup D.
\]
A third vertex $C \in V(G)$ is a common neighbor of $A$ and $B$ if and only if:
\begin{enumerate}
\item $C \cap D = \emptyset$ and $I \subset C$. In this case, $C$ is obtained from $I$ by
adding a single element.

\item $D \subset C$ and $C \subset U$. In this case, $C$ is obtained from $U$ by removing
a single element.
\end{enumerate}
(It is easy to see that in any other case, either $|A \triangle C| \neq 2$ or
$|B \triangle C| \neq 2$.)

If $C,C'$ are both common neighbors of $A$ and $B$, such that $C$ satisfies (1) and $C'$
satisfies (2), then clearly, $|C \triangle C'|=4$. Hence, if two common neighbors of $A$
and $B$ are themselves neighbors, then either both satisfy (1) or both
satisfy (2). On the other hand, it is clear that any two common neighbors satisfying (1)
are themselves neighbors, and the same holds for (2). Thus, any pair $A,B$ of adjacent vertices
of $G$ is included in at most two max-cliques:
\begin{enumerate}
\item $U(A,B) = \{C \in V(G) | C \subset A \cup B\}$, and

\item $I(A,B) = \{C \in V(G) | C \supset A \cap B\}$.
\end{enumerate}
For any two elements $C,C' \in U(A,B)$, the union $C \cup C'$ is constant (and equal
to $U$). Likewise, for any two elements $C,C' \in I(A,B)$, the intersection $C \cap C'$ is
constant (and equal to $I$). This is the motivation behind the following definition.
\begin{definition}
A max-clique of the first type will be called a Union max-clique,
or a \emph{U-clique}, and a max-clique of the second type will be called an Intersection
max-clique or an \emph{I-clique}.
\end{definition}

\begin{remark}\label{Remark:Three-determine}
It is clear that given two vertices $A,B$ of a max-clique $C$, we cannot determine whether
$C$ is a U-clique or an I-clique. However, if we are given a third vertex $B' \in C$, we can
determine the type, according to which one of the equalities $A \cap B = A \cap B'$,
$A \cup B = A \cup B'$ holds. Moreover, once we determine that $C$ is, say, an I-clique,
we know that $C=I(A,B)$ as this is the unique I-clique that includes $A$ and $B$. Hence,
three vertices of a max-clique determine it uniquely.
\end{remark}

\begin{definition}
If $U(A,B)=\{A,B\}$, we say that $\{A,B\}$ is a \emph{degenerate U-clique}. Similarly, if
$I(A,B)=\{A,B\}$, we say that $\{A,B\}$ is a \emph{degenerate I-clique}.
\end{definition}

For a pair of adjacent vertices $A,B$, there are four possible situations:
\begin{enumerate}
\item $|U(A,B)| \geq 3$ and $|I(A,B)| \geq 3$. In this case, there are exactly two max-cliques
that contain $A$ and $B$.

\item $|U(A,B)| \geq 3$ and $|I(A,B)| = 2$. In this case, there is a unique max-clique
that contains $A$ and $B$, namely $U(A,B)$. (The I-clique that contains $A$ and $B$ is degenerate.)

\item The same as (2), with the roles of $U(A,B)$ and $I(A,B)$ interchanged.

\item $U(A,B) = I(A,B) = \{A,B\}$. In this case, the set $\{A,B\}$ itself is a max-clique (that
is, both a U-clique and an I-clique). As we shall see in the sequel, this situation
cannot occur in our geometric setting.
\end{enumerate}

\subsection{Types of Max-Cliques in $G(P)$}
\label{sec:sub:geom-meaning}

Now we turn to the geometric graph $G(P)$ and show additional properties of max-cliques
that follow from its geometric structure. In order to make our notation suggestive, we denote
now the two adjacent vertices of $G(P)$ by $T_1,T_2$, and the edges in their symmetric difference
by $e_1 \in T_1 \setminus T_2$ and $e_2 \in T_2 \setminus T_1$.

In $G(P)$, U-cliques and I-cliques have a geometric meaning.
\begin{enumerate}
\item \textbf{$U(T_1,T_2)$.} Consider the graph $\bar{T} = T_1 \cup T_2 = T_1 \cup \{e_2\}$.
Obviously, it is a connected graph with a unique cycle. This cycle contains $e_1$ and $e_2$, and is simple
if and only if $e_1$ and $e_2$ do not cross. (Note that $e_2$ cannot cross another edge of $T_1$, as
both these edges belong to the SST $T_2$.) As shown above, $U(A,B)$ consists of
all vertices of $G(P)$ that are obtained from $\bar{T}$ by removing a single edge.
Since the vertices of $G(P)$ are the edge sets of SSTs, we can say that removing an edge
(other than $e_1,e_2$) from $\bar{T}$ results in an element of $U(T_1,T_2)$ if and only if
the unique cycle of $\bar{T}$ is simple, and the removed edge belongs to that cycle. Note that
if $e_1,e_2$ cross, then removal of any edge other than $e_1,e_2$ from $\bar{T}$ results in a
non-simple graph, and thus, the only elements of $U(T_1,T_2)$ are $T_1$ and $T_2$.

\item \textbf{$I(T_1,T_2)$.} Consider the graph $\tilde{T} = T_1 \cap T_2 = T_1 \setminus \{e_1\}$.
Obviously, it is a simple forest with two connected components. As shown above, $I(A,B)$ consists of
all vertices of $G(P)$ that are obtained from $\tilde{T}$ by adding a single edge. Since the
vertices of $G(P)$ are the edge sets of SSTs, we can say that adding an edge to $\tilde{T}$
results in an element of $I(A,B)$ if and only if that edge makes the forest $\tilde{T}$ into
a simple spanning tree of $K(P)$.
\end{enumerate}


\subsection{Geometric Characterization of Degenerate Cliques}
\label{sec:sub:char-degenerate}

The geometric interpretation allows us to characterize the cases when U-cliques and I-cliques are
degenerate.
\begin{claim}\label{Claim:Deg-U-clique}
Let $T_1,T_2$ be SSTs such that $T_1 \setminus T_2 = \{e_1\}$ and $T_2 \setminus T_1 = \{e_2\}$.
The U-clique $U(T_1,T_2)$ is degenerate if and only if $e_1$ and $e_2$ cross.
\end{claim}

\begin{proof}
By the geometric interpretation, if $e_1$ and $e_2$ cross then the only vertices
of $U(T_1,T_2)$ are $T_1$ and $T_2$, and thus, it is degenerate. If $e_1$ and $e_2$ do not cross,
then $\bar{T}$ is a {\it simple} connected graph on $n$ vertices with $n$ edges. (Note that $\bar{T}$
is simple since as $T_1,T_2$ are SSTs, the only edges in $\bar{T}$ that may cross each other
are $e_1,e_2$.) Thus, $\bar{T}$ has a cycle of order at least 3. Removal of any edge from this
cycle gives rise to a vertex in $U(T_1,T_2)$. Therefore, in this case $U(T_1,T_2)$
is non-degenerate.
\end{proof}

\begin{proposition}\label{Prop:Deg-I-clique}
Let $P$ be a finite set of points in the plane, no three on a line, $|P| \geq 4$. Let $T_1,T_2$ be SSTs of 
$K(P)$ such that $T_1 \setminus T_2 = \{e_1\}$ and $T_2 \setminus T_1 = \{e_2\}$.
The I-clique $I(T_1,T_2)$ is degenerate only in the following case:

The convex polygon $\mathrm{conv}(P)$ has three consecutive vertices $x,v,y$ (i.e., $[x,v]$ and
$[v,y]$ are edges of $\mathrm{conv}(P)$), such that:
\begin{enumerate}
\item The triangle $\mathrm{conv}(x,v,y)$ contains no other points of $P$.

\item $e_1 = [x,v]$ and $e_2 = [v,y]$.

\item $[x,y] \in T_1 \cap T_2$.
\end{enumerate}
\end{proposition}

\begin{proof}

Assume that~(1)--(3) hold, and let $T_0 \in I(T_1,T_2)$. $T_0$ is connected, and thus,
has an edge $e$ that emanates from $v$. Note that as $x,v,y$ are consecutive vertices
of $\mathrm{conv}(P)$ and the triangle $\mathrm{conv}(x,v,y)$ contains no other
points of $P$, any edge of $K(P)$ that emanates from $v$ (other than $e_1$ and $e_2$) must
cross $[x,y]$. Thus, either $e=e_1$, $e=e_2$, or $e$ crosses $[x,y]$. The latter is
impossible, since $[x,y] \in T_1 \cap T_2$, and as shown above,
$T_1 \cap T_2$ is included in any element of $I(T_1,T_2)$. If $e=e_1$, then
$T_1 = ((T_1 \cap T_2) \cup \{e_1\}) \subset T_0$, which implies $T_0=T_1$ (as both have
the same number of edges). Similarly, if $e=e_2$ then $T_0=T_2$. Therefore, the only
elements of $I(T_1,T_2)$ are $T_1$ and $T_2$, i.e., $I(T_1,T_2)$ is degenerate.

\begin{figure}[tb]
\begin{center}
\scalebox{0.8}{
\includegraphics{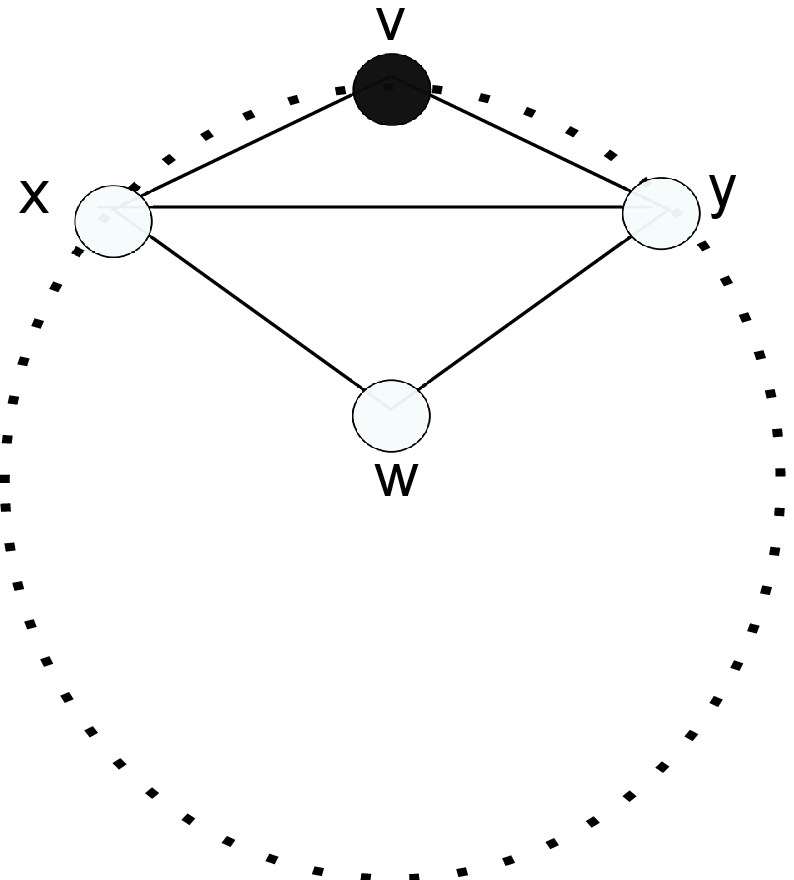}
} \caption{Illustration to the proof of Proposition~\ref{Prop:Deg-I-clique}.}\label{fig8}
\end{center}
\end{figure}

In the other direction, assume that $I(T_1,T_2)$ is degenerate. As mentioned
above, the graph $\tilde{T}=T_1 \cap T_2$ is a simple forest with two connected components.
Color the vertices of one component white and the vertices of the other component black,
and call an edge {\it colorful} if its endpoints are of different colors.

Since $\tilde{T}$ is planar, it can be extended to a triangulation $\mathcal{T}$ of
$\mathrm{conv}(P)$ with vertex set $P$. As $\mathcal{T}$ is connected, it contains a colorful edge $e$.
A triangle in $\mathcal{T}$ to which $e$ belongs clearly contains
another colorful edge $e'$. Addition of either $e$ or $e'$ to $\tilde{T}$ results in a simple tree,
and thus, gives rise to a vertex of $I(T_1,T_2)$. (Note that $e$ and $e'$ cannot cross edges of $\tilde{T}$ since
they belong to a triangulation that extends $\tilde{T}$.) Since $I(T_1,T_2)$ is degenerate, this implies that all
other edges of $\mathcal{T}$ are not colorful. We claim that this can happen only in the case described in the
statement of the proposition.

Consider the edges $e,e'$. If $e$ is not a boundary edge of $\mathrm{conv}(P)$ then it belongs to another triangle
in $\mathcal{T}$. The other triangle must contain an additional colorful edge, contradicting the assumption
that only $e$ and $e'$ are colorful. The same holds for $e'$, and thus, both $e$ and $e'$ are boundary edges of
$\mathrm{conv}(P)$. Denote their common vertex by $v$ and their other endpoints by $x,y$, respectively.

It is clear that Condition~(1) above holds for $x,v,y$, since $\mathcal{T}$ is a triangulation of $P$ and
$\mathrm{conv}(v,x,y)$ is one of its triangles. To see that Condition~(2) holds, note that $\tilde{T} \cup \{e\}$ and
$\tilde{T} \cup \{e'\}$ are the only vertices of $I(T_1,T_2)$, and thus, are equal to $T_1$ and $T_2$. As
$T_i=\tilde{T} \cup \{e_i\}$ for $i=1,2$, the edges $e,e'$ must coincide with $e_1$ and $e_2$.

Finally, since $|P| \geq 4$, $[x,y]$ is a diagonal of $\mathrm{conv}(P)$, and thus, in the triangulation
$\mathcal{T}$ it belongs to another triangle $\mathrm{conv}(x,y,w)$ (see Figure~\ref{fig8}).
This triangle is monochromatic (as
otherwise, there are at least four colorful edges in $\mathcal{T}$), hence $w,v$ are of different colors.
We apply a flip to the triangulation $\mathcal{T}$, replacing the triangles $\mathrm{conv}(x,y,v),
\mathrm{conv}(x,y,w)$ by $\mathrm{conv}(x,v,w),\mathrm{conv}(y,v,w)$. The resulting triangulation includes
an additional colorful edge $[v,w]$ that does not cross any edge of $\tilde{T}$, except possibly for $[x,y]$.
Since by assumption, there are only two colorful edges that do not cross edges of $\tilde{T}$, we
must have $[x,y] \in \tilde{T}$, which means that Condition~(3) holds.

\end{proof}

\begin{corollary}\label{Cor:At-least-one-non-degenerate}
Each edge of $G(P)$ is contained in at least one non-degenerate max-clique.
\end{corollary}

\begin{proof}
Let $[T_1,T_2] \in G(P)$\footnote{For sake of clarity, we use here and in the sequel the notation $[T,T']$ for
edges of $G(P)$, like is commonly used for geometric graphs, although $G(P)$ is treated as an abstract graph.} and denote $T_1 \setminus T_2 = \{e_1\}$ and $T_2 \setminus T_1 = \{e_2\}$.
If $e_1,e_2$ do not cross then $U(T_1,T_2)$ is non-degenerate by
Claim~\ref{Claim:Deg-U-clique}. If $e_1,e_2$ cross then $I(T_1,T_2)$ is
non-degenerate, since by the proof of Proposition~\ref{Prop:Deg-I-clique}, if $I(T_1,T_2)$ is degenerate, then
the only edges that can be added to $T_1 \cap T_2$ to form a simple tree share a vertex, which is not the case for $e_1,e_2$ that cross in an interior point.
\end{proof}

\subsection{Identification of the type of a Degenerate Clique in $G(P)$}
\label{sec:sub:identify-degenerate}

\begin{lemma}\label{Lemma:I-clique-not-leaf}
Let $T$ be a SST of $G(P)$, and let $D$ be an I-clique that contains $T$. Denote the
common intersection of pairs of elements of $D$ by $\tilde{T}$, and let
$\{e\}=T \setminus \tilde{T}$. If $e$ is not a leaf edge of $T$, then $|V(D)| \geq 4$.
\end{lemma}

\begin{proof}
We begin the proof with the argument used in the proof of Proposition~\ref{Prop:Deg-I-clique}.
Namely, we consider the graph $\tilde{T}$, which is a simple forest with two components.
We color its components black and white, and extend it to a triangulation $\mathcal{T}$
of $P$. The triangulation contains a colorful edge $[a,b]$, and consequently, another
colorful edge $[a,b']$ in the same triangle. Assume w.l.o.g. that $a$ is
black, $b$ and $b'$ are white.

Now we would like to use the assumption that $e$ is not a leaf edge of $T$. This assumption
implies that each connected component of $\tilde{T}$ has at least two vertices.
Consequently, any vertex $v \in P$ is an endpoint of at least one monochromatic edge of
$\mathcal{T}$ (as otherwise, $v$ would be isolated in $\tilde{T}$).

If both $[a,b]$ and $[a,b']$ are boundary edges of $\mathrm{conv}(P)$, then (since
$\triangle a,b,b'$ is a triangle in $\mathcal{T}$) these are the only edges of $\mathcal{T}$ that
emanate from $a$. This is impossible, as they are both colorful. Hence, we can assume
w.l.o.g. that $[a,b]$ is a diagonal of $\mathrm{conv}(P)$, and thus, belongs to
two triangles, $\triangle abc$ and $\triangle abd$. (Note that either $c$ or $d$ is equal
to the vertex $b'$ mentioned above.) We consider several cases, according to the colors of
$c$ and $d$:


\begin{enumerate}
\item \textbf{Case 1: $c$ and $d$ are white.} (This case is illustrated in Figure~\ref{fig2p2}.)
Consider the neighbors of $a$ in $\mathcal{T}$.
Since $\mathcal{T}$ is a triangulation, all these vertices lie on a path in $\mathcal{T}$.
As stated before, since $a$ is black, at least one of its neighbors must be black. On the
other hand, some of its neighbors (including $b,c,d$) are white. Hence, at least one of the
edges in the path connecting the neighbors of $a$ is colorful (see Figure~\ref{fig2p2}).
In addition, the edges $[a,b],[a,c],[a,d]$ are colorful. Thus, $\mathcal{T}$ contains at least four colorful edges,
and each of them gives rise to a vertex of $D$. Therefore, $|V(D)| \geq 4$.

\begin{figure}[tb]
\begin{center}
\scalebox{0.8}{
\includegraphics{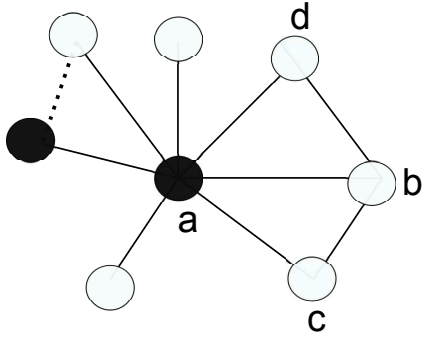}
} \caption{Illustration to Case~1 of Lemma~\ref{Lemma:I-clique-not-leaf}.} \label{fig2p2}
\end{center}
\end{figure}

\item \textbf{Case 2: $c$ and $d$ are black.}
Since the black vertices $a,c,d$ are all neighbors
of the white vertex $b$, the argument of Case 1 applies, with the roles of $a$ with $b$,
black and white, interchanged.

\item \textbf{Case 3: $c$ and $d$ have different colors.} Assume w.l.o.g. that $c$ is white
and $d$ is black. In this case, the edges $[a,b],[a,c],$ and $[b,d]$ are colorful. We further
divide this case into three subcases, according to whether $[a,c]$ and
$[b,d]$ are diagonals of $\mathrm{conv}(P)$ or not, and show that in each case,
$\mathcal{T}$ contains at least one additional colorful edge.
\begin{enumerate}
\item \textbf{Case 3a: $[b,d]$ is a diagonal of $\mathrm{conv}(P)$.} In this case, $[b,d]$
belongs to an additional triangle of $\mathcal{T}$ (i.e., other than $\triangle abd$). If this
triangle is $\triangle bdc$, then the edge $[c,d]$ is colorful, implying $|V(D)| \geq 4$ (see
left part of Figure~\ref{fig3p2}). If
this triangle is $\triangle bdk$ for some $k \neq c$, then, as $b$ and $d$ have different colors,
one of the edges $[b,k]$ and $[d,k]$ is colorful, again implying $|V(D)| \geq 4$ (see right part of
Figure~\ref{fig3p2}).

\begin{figure}[tb]
\begin{center}
\scalebox{0.8}{
\includegraphics{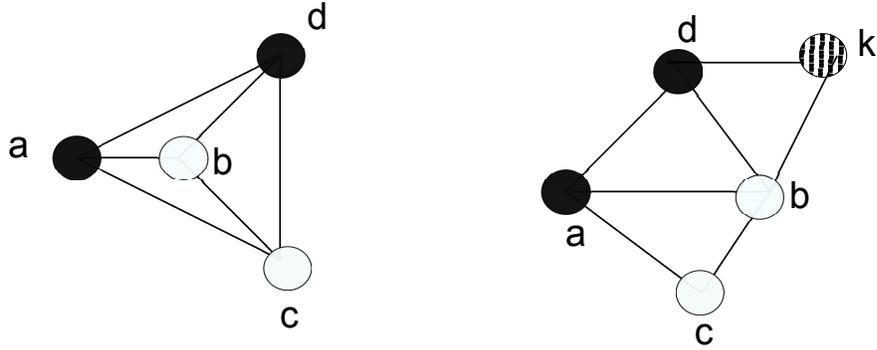}
} \caption{Illustrations to Case~3a of Lemma~\ref{Lemma:I-clique-not-leaf}} \label{fig3p2}
\end{center}
\end{figure}


\item \textbf{Case 3b: $[a,c]$ is a diagonal of $\mathrm{conv}(P)$.} The argument
of Case 3a applies, with the roles of $a$ and $b$, $c$ and $d$, black and white,
interchanged.

\item \textbf{Case 3c: Both $[a,c]$ and $[b,d]$ are boundary edges of $\mathrm{conv}(P)$.} In this
case, $\square abcd$ is a convex quadrilateral, and thus, its diagonal $[c,d]$ lies inside it
(see Figure~\ref{fig5p2}).
As both $\triangle abc$ and $\triangle abd$ are triangles in $\mathcal{T}$, they do not contain
points of $P$, hence $[c,d]$ does not cross any edge of $\tilde{T}$. (Note that $[a,b]$
is not an edge of $\tilde{T}$, since it is colorful.) Since $[c,d]$ is colorful, the graph
$\tilde{T} \cup \{[c,d]\}$ is an SST, hence belongs to $D$. Therefore, $|V(D)| \geq 4$, which
completes the proof of the lemma.

\begin{figure}[tb]
\begin{center}
\scalebox{0.8}{
\includegraphics{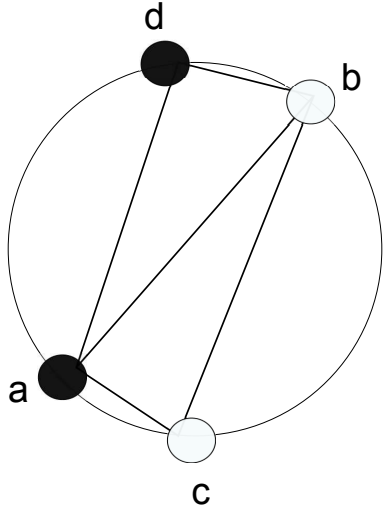}
} \caption{Illustration to Case 3c of Lemma~\ref{Lemma:I-clique-not-leaf}.} \label{fig5p2}
\end{center}
\end{figure}

\end{enumerate}
\end{enumerate}
\end{proof}

Using the lemma, we can identify whether a given degenerate clique is a U-clique
or an I-clique.
\begin{proposition}\label{Prop:Identify-Degenerate}
Let $[S,T]$ be an edge in $G(P)$ that constitutes a degenerate clique. (This actually
means that there is only one max-clique that includes $\{S,T\}$.) If $[S,T]$
is included in a max-clique of 3 vertices, then the degenerate clique
$[S,T]$ is an I-clique. If $[S,T]$ is included in a max-clique of at least 4 vertices,
then the degenerate clique $[S,T]$ is a U-clique.
\end{proposition}

\begin{proof}
By Corollary~\ref{Cor:At-least-one-non-degenerate}, each edge is included in at least
one non-degenerate max-clique. Hence, $[S,T]$ is included in a max-clique of size at
least 3.

If the degenerate clique $[S,T]$ is an I-clique, then, by Proposition~\ref{Prop:Deg-I-clique},
$S \cap T$ is a forest with two connected components. One of them is
an isolated vertex $v$ that lies on the boundary of $\mathrm{conv}(P)$, and the two neighbors
of $v$ on the boundary of $\mathrm{conv}(P)$, $x,y$, are adjacent in $S \cap T$. In such a case, the unique
cycle in the graph $S \cup T$ is the triangle $\triangle (x,v,y)$. By the geometric
characterization of Section~\ref{sec:sub:geom-meaning}, this implies that the size of the
U-clique that includes $[S,T]$ is 3.

If the degenerate clique $[S,T]$ is a U-clique, then, by Claim~\ref{Claim:Deg-U-clique},
the unique cycle of $S \cup T$ includes the edges $s \in S \setminus T$ and $t \in T \setminus S$,
and these edges cross. This implies that $s$ is not a leaf edge of $S$, and thus, by
Lemma~\ref{Lemma:I-clique-not-leaf}, the size of the I-clique that includes $[S,T]$ is at
least 4.
\end{proof}

\subsection{Identification of the type of a Non-Degenerate Max-clique in $G(P)$}
\label{sec:sub:identify-non-degenerate}

Our next goal is to identify whether a given non-degenerate max-clique is a U-clique or an
I-clique. This identification is somewhat more complex, and requires some preparations.
\begin{definition}
For any SST $T \in G(P)$, define a graph $D_T$ as follows: $V(T)$ is the set of max-cliques
that contain $T$ (including degenerate cliques). Two max-cliques are adjacent in $D_T$
if and only if their intersection is a single edge (that obviously has $T$ as one of its endpoints).
\end{definition}
Note that by Remark~\ref{Remark:Three-determine}, two non-identical max-cliques may intersect
in at most two vertices (since three mutually adjacent vertices determine a max-clique uniquely).
As all vertices of $D_T$ include $T$, it follows that any two non-adjacent vertices of $D_T$ intersect
in $T$ only.
\begin{theorem}\label{Thm:D_T}
Assume that $|P| \geq 5$. For any $T \in G(P)$, the graph $D_T$ is connected.
\end{theorem}

Before we present the proof of the theorem, we show how it can be used (along with several
of the previous lemmas) to identify the types of max-cliques in $G(P)$.

\medskip

By the discussion in Section~\ref{sec:sub:max-cliques-basics}, each edge of $G(P)$ belongs
to exactly one U-clique and exactly one I-clique (one of them possibly degenerate). This
implies that $D_T$ is 2-colorable. Indeed, we can color all its vertices that are U-cliques
white and all its vertices that are I-cliques black. If two vertices are adjacent,
both include the same edge $[S,T] \in G(P)$, and thus, one is a U-clique and
the other is an I-clique, so they have different colors. Using Theorem~\ref{Thm:D_T}, we can
conclude that $D_T$ is connected and 2-colorable, which implies that the 2-coloring is
unique, in the sense that fixing the color of any vertex determines the colors of all
other vertices. This will allow us to determine the types of all max-cliques, by the
following four-step process:
\begin{enumerate}
\item Pick an SST $T$ that belongs to a degenerate clique $\{S,T\}$. (It is easy to show that
such a $T$ always exists. We show this in Lemma~\ref{Lemma:Exist-degenerate} below.) Using
Proposition~\ref{Prop:Identify-Degenerate}, identify whether $\{S,T\}$ is a U-clique
or an I-clique.

\item Consider the vertices of $D_T$ and determine for each of them whether it is a
U-clique or an I-clique. (This is possible by the explanation above. Since we know the
``color'' of the vertex $\{S,T\}$ of $D_T$, we can determine the colors of all
other vertices.)

\item Consider a neighbor $T'$ of $T$. Note that the edge $[T,T'] \in G(P)$ belongs to
a max-clique $C$ that is a vertex of both $D_T$ and $D_{T'}$. Determine whether $C$ is
a U-clique or an I-clique. (This is possible, as by the previous step we can determine
the type for all max-cliques that are vertices of $D_T$.) Using this information about
$C \in V(D_{T'})$, determine for each vertex of $D_{T'}$ whether it is a U-clique or
an I-clique.

\item Repeat Step~3 with a ``new'' vertex of $G(P)$ every time until the types of all
max-cliques are determined. (Since $G(P)$ is connected by Claim~\ref{Claim:G(P)-connected},
we indeed reach all vertices of $G(P)$ in this way.)
\end{enumerate}

Therefore, in order to determine the types of all max-cliques we have to prove a simple
lemma on the existence of degenerate max-cliques, and to prove Theorem~\ref{Thm:D_T}.
We provide these two items now.
\begin{lemma}\label{Lemma:Exist-degenerate}
For any set $P$, $|P| \geq 5$, of points in general position in the plane, there exists
an SST that belongs to a degenerate U-clique in $G(P)$.
\end{lemma}

\begin{proof}
By the Erd\H{o}s-Szekeres theorem~\cite{ES35}, there exist four points in $P$ in convex
position. Denote these points $a,b,c,d$, such that $[a,b,c,d]$ is a convex quadrilateral
(in this order).
Consider the tree $T$ that includes the edges $[a,b],[b,d],[d,c]$, all edges that connect
$a$ to each $p \in P$ that lies above the straight line $\ell(b,d)$ and all edges that connect
$c$ to each $p \in P$ that lies below $\ell(b,d)$, as shown in Figure~\ref{fig6p4}. Clearly,
$T$ is an SST of $K(P)$. Addition of the edge $[a,c]$ to $T$ creates a self-crossing cycle, and thus,
by the discussion in Section~\ref{sec:sub:geom-meaning}, $T \in V(G)$ belongs to a degenerate
U-clique.

\begin{figure}[tb]
\begin{center}
\scalebox{0.8}{
\includegraphics{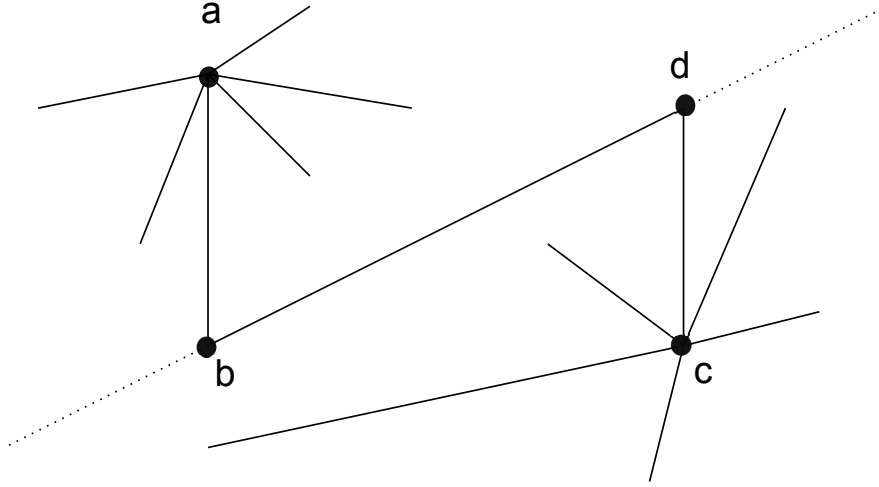}
} \caption{An illustration for the proof of Lemma~\ref{Lemma:Exist-degenerate}.} \label{fig6p4}
\end{center}
\end{figure}

\end{proof}

\medskip

\begin{proof}[Proof of Theorem~\ref{Thm:D_T}]
Suppose $C,C' \in V(D_T)$. Choose edges $[T,S] \in C$, $[T,S'] \in C'$. In order to prove
that $C$ is connected to $C'$ by a path in $D_T$, it suffices to find a sequence $S^0,S^1,\ldots,S^\ell$
of SST's, such that $S^0=S$, $S^\ell =S'$, and for each $0 \leq i< \ell$, there is a max-clique
$C_i$ of $G(P)$ that includes both $[T,S^i]$ and $[T,S^{i+1}]$. This will be done in four steps.
\begin{enumerate}
\item \textbf{Step 1.} Extend the SST $T$ to a triangulation $\mathcal{T}$ of $\mathrm{conv}(P)$.
Recall that $\mathcal{T}$ is a 2-connected graph.

\item \textbf{Step 2.} If $S \subset \mathcal{T}$, leave $S$ as it is. If not, $\mathcal{T}$
contains the graph $\tilde{T}= T \cap S = T \setminus \{e\}$. Since $\mathcal{T}$ is 2-connected,
there is another edge $e^{*}$ in $\mathcal{T}$ that connects the two components of $\tilde{T}$.
Define $S^{*} = \tilde{T} \cup \{e^{*}\}$. $S^{*}$ is an SST of $K(P)$. Note that $S^{*}$ is
included in the I-clique $I(T,S)$ (since $T \cap S^{*} = T \cap S = \tilde{T}$).

Do the same for $S'$: leave it, if $S' \subset \mathcal{T}$, or replace it by some
$S^{'*} \subset \mathcal{T}$, such that $S^{'*} \in I(T,S')$.

Step~2 allows us to restrict our attention to the case where both $S$ and $S'$ are included
in $\mathcal{T}$ (and all intermediate SST's will be included in $\mathcal{T}$, as well).

\item \textbf{Step 3.} Assume $T \cap S = T \setminus \{e\}$, $T \cap S' = T \setminus \{e'\}$.
If $e=e'$, then $[T,S]$ and $[T,S']$ belong to the same I-clique. Suppose $e \neq e'$. Since
$T$ is a simple tree, there is a unique simple path in $T$ whose edges are (in this
order) $e_1,e_2,\ldots,e_m$, with $e_1=e$ and $e_m = e'$. For $1 <i <m$, choose an edge
$e^*_i$ of $\mathcal{T}$ other than $e_i$ that connects
the two components of $T \setminus \{e_i\}$, and define $S_i = T \setminus \{e_i\} \cup \{e^*_i\}$.
In addition, let $S_1=S$ and $S_m=S'$. Now we only have to connect $S_{i-1}$ with $S_i$
for $i=2,3,\ldots,m$. (Note that the desired sequence $S^0,S^1,\ldots,S^\ell$ will be the
concatenation of the sequences connecting $S_1$ to $S_2$, $S_2$ to $S_3$ etc., and $\ell$
will be the sum of their lengths.)

\item \textbf{Step 4.} Suppose $\mathcal{T}$ is a triangulation of $\mathrm{conv}(P)$
with vertex set $P$. Let $T,S,S'$ be SST's of $K(P)$ that are included in $\mathcal{T}$.
Assume that both $S$ and $S'$ are adjacent to $T$ in $G(P)$,  $S \cap T  = T \setminus \{e\}$,
$S' \cap T = T \setminus \{e'\}$, $e \neq e'$, and $e,e'$ share a vertex. Suppose $e=[x,y], e'=[x',y]$,
and let $k=\deg(y)-2$. Denote the edges of $T$ that emanate from
$y$ by $e,e',f_1,f_2,\ldots,f_k$. Removal of all edges that emanate from $y$ divides $T$ into
$k+3$ connected components: $D$ that includes $x$, $D'$ that includes $x'$, $F_s$
($1 \leq s \leq k$) that includes the second endpoint of $f_s$,
and $F_{k+1}$ that consists of the isolated vertex $y$. Since $\mathcal{T}$ is 2-connected, we
can extend the forest $B= D \cup D' \cup F_1 \cup \ldots \cup F_k$ into an SST of $K(P \setminus \{y\})$
by adding $k+1$ edges of $\mathcal{T}$.

Let $\bar{S}$ be such an extension. $\bar{S}$ includes a unique simple path $\pi$ from $x$ to
$x'$. This path starts in the component $D$ of $B$, and ends in $D'$. It visits some of the
intermediate components $F_i$ in a particular order (see Figure~\ref{fig8p6}). By appropriately labelling
these components, we may assume that $\pi$ visits $D,F_1,F_2,\ldots,F_l,D'$ in this order ($0 \leq l \leq k$).
Let $g_0$ be the edge of $\pi$ that passes from $D$ to $F_1$, $g_i$ ($1 \leq i \leq l-1$) the edge
that passes from $F_i$ to $F_{i+1}$, and $g_l$ be the edge that passes from $F_l$ to $D'$. (If $l=0$,
then there is only one edge $g=g_0$ that passes directly from $D$ to $D'$.)

Now we can describe the passage from $[T,S]$ to $[T,S']$. Let us start with the simple case $l=0$.
Put $S_1=T \setminus \{e\} \cup \{g\}$, $S_2 = T \setminus \{e'\} \cup \{g\}$. Then
$T \cap S = T \cap S_1$, $T \cup S_1 = T \cup S_2$, and $T \cap S_2 = T \cap S'$. Thus, $\{T,S,S_1\}$
and $\{T,S_2,S'\}$ are included in I-cliques, while $\{T,S_1,S_2\}$ is included in a U-clique. Hence,
$(S,S_1,S_2,S')$ is the required sequence.

When $l>0$, define
\begin{align*}
&S_1=T \setminus \{e\} \cup \{g_0\}, \qquad S_2=T \setminus \{f_1\} \cup \{g_0\}, \\
&S_3=T \setminus \{f_1\} \cup \{g_1\}, \qquad S_4=T \setminus \{f_2\} \cup \{g_1\}, \\
&\ldots \\
&S_{2l-1}=T \setminus \{f_{l-1}\} \cup \{g_{l-1}\}, \qquad S_{2l}=T \setminus \{f_l\} \cup \{g_{l-1}\}, \\
&S_{2l+1}=T \setminus \{f_l\} \cup \{g_l\}, \qquad \quad S_{2l+2}=T \setminus \{e'\} \cup \{g_l\}.
\end{align*}
Then each of the triples $\{T,S,S_1\},\{T,S_{2i},S_{2i+1}\}$ (for $1 \leq i \leq l$), and $\{T,S_{2l+2},S'\}$
is included in an I-clique, and each of the triples $\{T,S_{2i-1},S_{2i}\}$ (for $1 \leq i \leq l+1$) is
included in a U-clique. Therefore, $(S,S_1,S_2,\ldots,S_{2l+1},S_{2l+2},S')$ is the required sequence.
This completes the proof of the theorem.

\begin{figure}[tb]
\begin{center}
\scalebox{0.8}{
\includegraphics{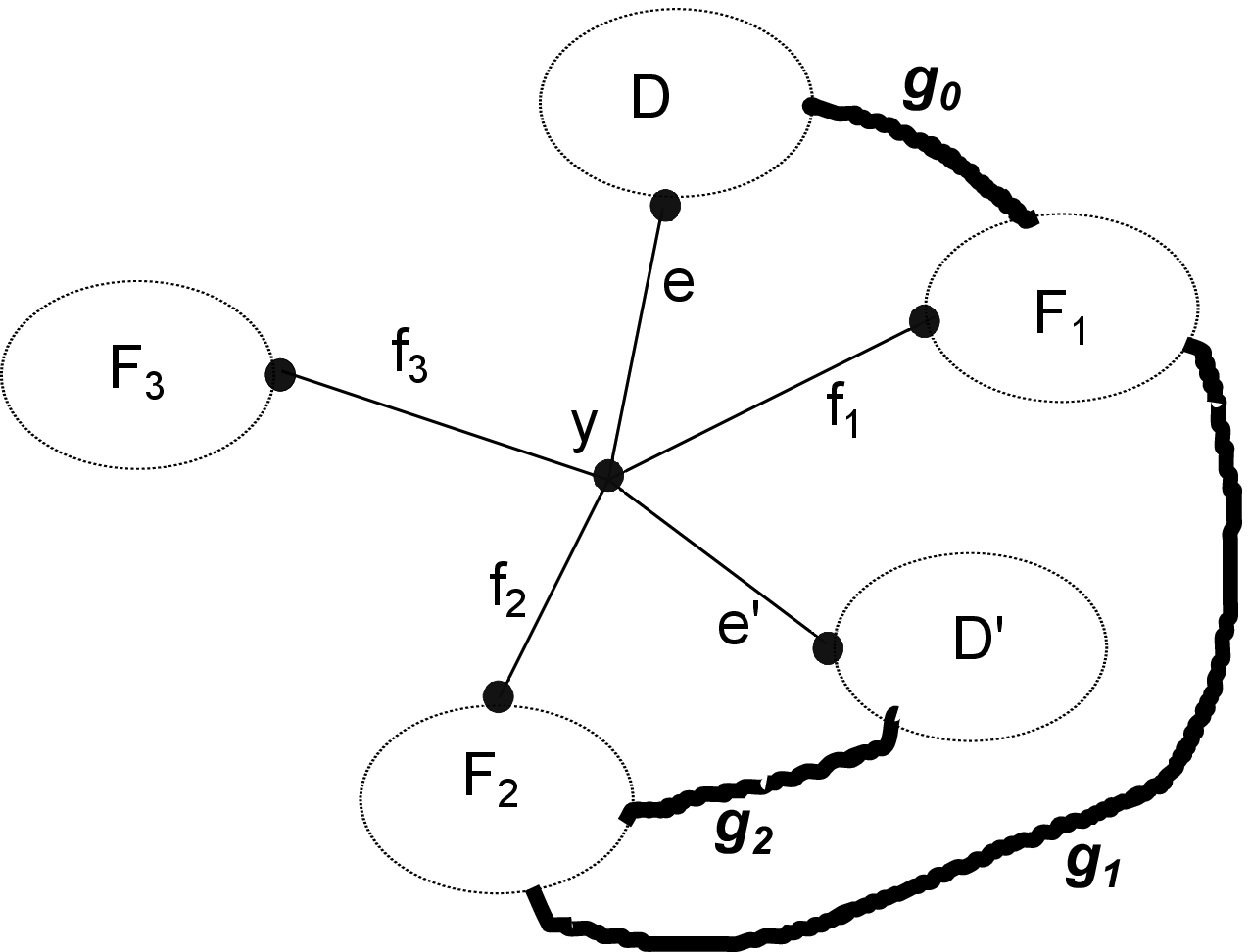}
} \caption{An illustration for the proof of Theorem~\ref{Thm:D_T}.}
\label{fig8p6}
\end{center}
\end{figure}
\end{enumerate}

\end{proof}

As explained after the statement of Theorem~\ref{Thm:D_T}, the Theorem and
Lemma~\ref{Lemma:Exist-degenerate} imply the following corollary.
\begin{corollary}
Assume $|P| \geq 5$. Given $G(P)$, we can determine for each max-clique in it whether it is a U-clique 
or an I-clique.
\end{corollary}

\section{Identification of Stars and Brushes}
\label{sec:stars-and-brushes}

In this section we use the results of Section~\ref{sec:max-cliques} to identify the vertices of
$G(P)$ that represent stars (i.e., SSTs of diameter 2) and brushes (i.e., SSTs of diameter 3). The
stars, considered in Section~\ref{sec:sub:stars}, are determined only up to an automorphism of $K(P)$
as a geometric graph. The brushes, considered in Section~\ref{sec:sub:brushes}, are determined uniquely
given a determination of the stars.

\subsection{Identification of Stars}
\label{sec:sub:stars}

\begin{definition}
A \emph{star} is a tree of diameter 2.
\end{definition}

\begin{notation}
For $x \in P$, we call the spanning star whose center is $x$ an \emph{$x$-star},
and denote it by $S(x)$.
\end{notation}

\begin{theorem}\label{Thm:Stars}
A vertex $T \in G(P)$ is a star if and only if all U-cliques that include $T$ are
of size 3.
\end{theorem}

\begin{proof}
Recall that by the geometric interpretation
of max-cliques presented in Section~\ref{sec:sub:geom-meaning}, if $U(S,T)$ is a non-degenerate
U-clique then all its elements are obtained from the graph $S \cup T$ by removing an edge from
its unique cycle. In particular, $|U(S,T)|$ is the length of the unique cycle of $S \cup T$. If
the unique cycle of $S \cup T$ is self-crossing (which can occur only if its length is $\geq 4$)
then $U(S,T)$ is degenerate.

Assume that $T \in G(P)$ is an $x$-star, and let $U(S,T)$ be a U-clique that includes $T$. Since
$T$ is a star, $S \cup T$ is obtained from $T$ by adding an edge that connects two leaves $v,w$
of $T$. Hence, the unique cycle of $S \cup T$, $([x,v],[v,w],[w,x])$, is of length 3. Thus,
$|U(S,T)|=3$, as asserted.

On the other hand, we show that if $T \in G(P)$ is an SST of diameter $\geq 3$, then $T$
belongs to a U-clique of size $\neq 3$. Since $\mathrm{diam}(T) \geq 3$, $T$ contains an
internal edge $[a,b]$ (i.e., both $a$ and $b$ are not leaves of $T$). Consider the graph
$T \setminus \{[a,b]\}$, that is obviously a forest with two connected components. Color the
vertices of the connected component that includes $a$ black and the vertices of the component
that includes $b$ white.

We would like to show that there exists a colorful edge $e$ that uses neither $a$ nor $b$ and 
does not cross any edge of $T \setminus \{[a,b]\}$ (see
Figure~\ref{fig9p7}). This will
conclude the proof, since in such a case, denoting $S = T \cup \{e\} \setminus \{[a,b]\}$,
we find that $S$ is an SST, $[S,T] \in G(P)$, and the unique cycle of the graph
$S \cup T = T \cup \{e\}$ is of length $\geq 4$. Then, by the geometric interpretation above,
if $e$ crosses $[a,b]$ then $U(S,T)$ is a degenerate U-clique, and otherwise, $|U(S,T)|$ is
equal to the length of the unique cycle of $S \cup T$, that is $\geq 4$. Hence, in any case,
$T$ lies in a U-clique of size $\neq 3$.

\begin{figure}[tb]
\begin{center}
\scalebox{0.8}{
\includegraphics{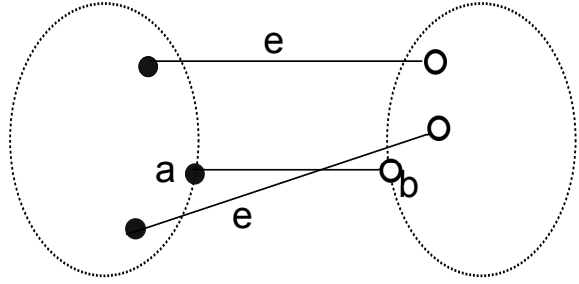}
} \caption{An illustration for the proof of Theorem~\ref{Thm:Stars} -- beginning.}
\label{fig9p7}
\end{center}
\end{figure}

We extend $T$ to a triangulation $\mathcal{T}$ of $\mathrm{conv}(P)$, and consider three
cases:
\begin{enumerate}
\item \textbf{Case 1: $[a,b]$ is a boundary edge of $\mathrm{conv}(P)$.} The edge $[a,b]$, being
colorful, is contained in a colorful triangle $\triangle abc \in \mathcal{T}$. The other
colorful edge of $\triangle abc$ is not a boundary edge of $\mathrm{conv}(P)$, as
otherwise, one of the two connected components of $T \setminus \{[a,b]\}$ consists of a
single vertex, which contradicts the assumption that $[a,b]$ is an internal edge of $T$.
Denote that colorful edge $[a,c]$. The neighbors of $a$ in $\mathcal{T}$ constitute a path
$[b,c,d_1,d_2,\ldots]$. Since the connected component of $a$ in the graph $T \setminus \{[a,b]\}$
(i.e., the ``black'' component) includes more than one vertex, at least one of the $d_i$'s
is black. Thus, the path contains a colorful edge $e$ that uses neither $a$ nor $b$
(see Figure~\ref{fig10p8}). Furthermore, as $e$ belongs to $\mathcal{T}$, it does not
cross any edge of $T$. Hence, $e$ is the edge whose existence was claimed.

\begin{figure}[tb]
\begin{center}
\scalebox{0.8}{
\includegraphics{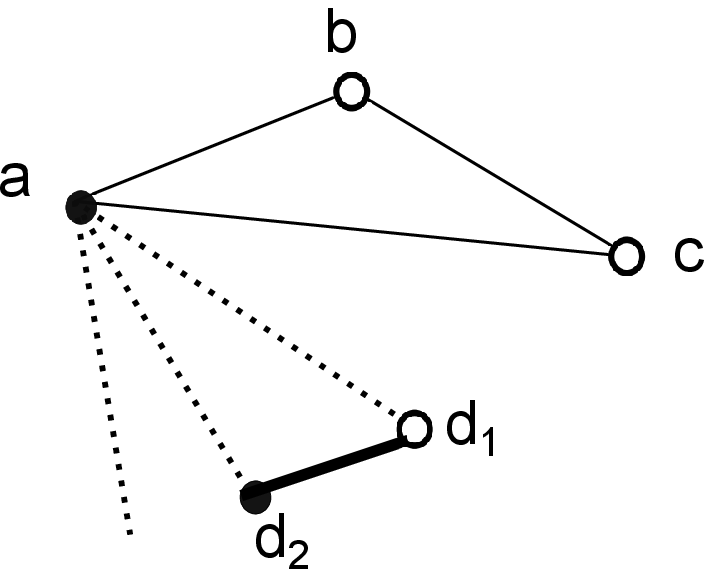}
} \caption{An illustration for the proof of Theorem~\ref{Thm:Stars} -- case 1.}
\label{fig10p8}
\end{center}
\end{figure}

\item \textbf{Case 2: $b$ is an internal vertex of $\mathrm{conv}(P)$.} In this case, $[a,b]$ is an
internal edge of $\mathcal{T}$, and thus, it is contained in two triangles $\triangle abc,
\triangle abd \in \mathcal{T}$. We further divide this case into two sub-cases:

\begin{enumerate}
\item \textbf{Case 2a: Either $c$ or $d$ (or both) are black.} Assume w.l.o.g. that $c$ is black.
Since $b$ is an internal vertex of $\mathrm{conv}(P)$, the neighbors of $b$ in $\mathcal{T}$
constitute a cycle $[d,a,c,d_1,d_2,\ldots,d_k,d]$. Since the connected component of $b$ in
the graph $T \setminus \{[a,b]\}$ contains more than one vertex, at least one of the $d_i$'s
or $d$ is white. Since $a,c$ are black, this implies that the cycle includes at least two colorful
edges, and at least one of them uses neither $a$ nor $b$ (see left part of Figure~\ref{fig11p8}).
As in Case 1, this is the desired edge $e$.

\begin{figure}[tb]
\begin{center}
\scalebox{0.8}{
\includegraphics{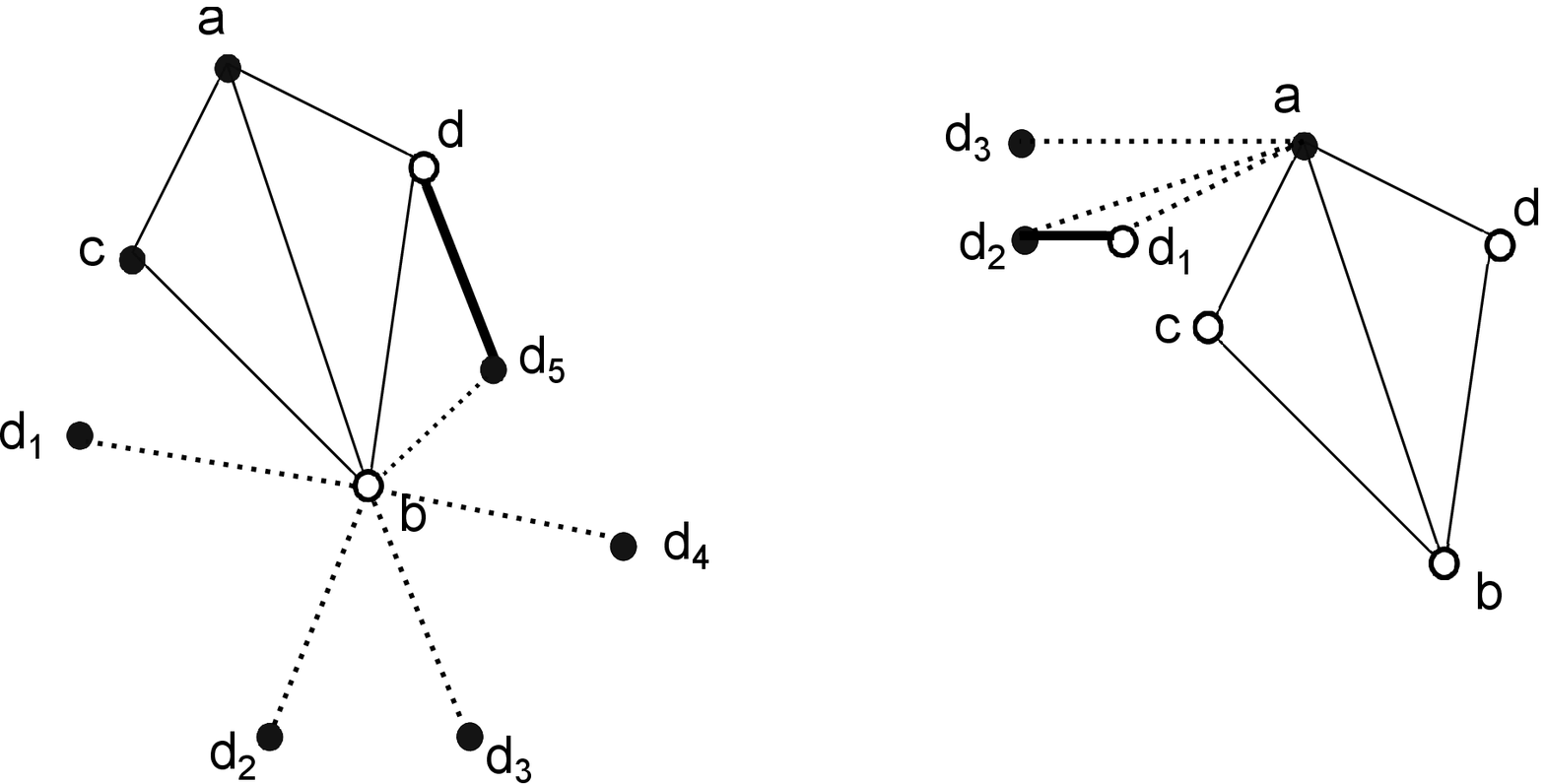}
} \caption{An illustration for the proof of Theorem~\ref{Thm:Stars} -- case 2.}
\label{fig11p8}
\end{center}
\end{figure}

\item \textbf{Case 2b: Both $c$ and $d$ are white.} By the same argument as above, $a$ has a black
neighbor in $\mathcal{T}$. As the neighbors of $a$ in $\mathcal{T}$ form a (possibly closed) path, at least
one edge of this path is colorful (see right part of Figure~\ref{fig11p8}). This edge uses neither $a$ nor $b$ (as both edges that
use $b$ in this path, $[b,c]$ and $[b,d]$, are not colorful). As in the previous cases, this is
the desired edge $e$.


\end{enumerate}

\item \textbf{Case 3: Both $a$ and $b$ are boundary vertices of $\mathrm{conv}(P)$, and $[a,b]$ is
a diagonal of $\mathrm{conv}(P)$.} As in Case 2, $[a,b]$ lies in two triangles $\triangle abc,
\triangle abd \in \mathcal{T}$. We further divide this case to two sub-cases:

\begin{enumerate}
\item \textbf{Case 3a: $c$ and $d$ are of the same color.} W.l.o.g., $c$ and $d$ are white. By the same
arguments as above, $a$ has a black neighbor, and thus, the path of $a$'s neighbors includes a
colorful edge (see left part of Figure~\ref{fig13p9}). This edge does not use $b$, as both edges that use $b$
(that are $[b,c],[b,d]$) are not colorful, and it clearly does not use $a$.

\begin{figure}[tb]
\begin{center}
\scalebox{0.8}{
\includegraphics{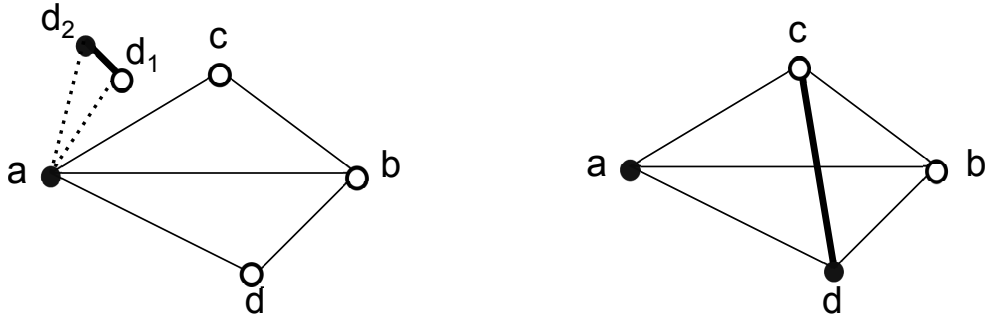}
} \caption{An illustration for the proof of Theorem~\ref{Thm:Stars} -- Case 3.}
\label{fig13p9}
\end{center}
\end{figure}

\item \textbf{Case 3b: $c$ and $d$ are of different colors.} In this case, the edge $[c,d]$ is as
desired, since it is colorful, does not use $a,b$, and does not cross edges of $T \setminus \{[a,b]\}$
(see right part of Figure~\ref{fig13p9}). Note that since $[c,d]$ crosses $[a,b]$, the U-clique that includes
$T$ in this case is degenerate. This completes the proof of the theorem.


\end{enumerate}
\end{enumerate}
\end{proof}

\begin{remark}
It may happen that $T \in G(P)$ is not a star, but all U-cliques that contain $T$ are
either degenerate or of size 3.
\end{remark}


\subsection{Identification of Brushes in $G(P)$}
\label{sec:sub:brushes}

\begin{definition}
Let $P$ be a set of points in general position in the plane, and let $p,q \in P$. A
\emph{$pq$-brush} is an SST of diameter 3 whose only internal edge is $[p,q]$.
\end{definition}

Figure~\ref{fig15p9} shows an example of a $pq$-brush.

\begin{figure}[tb]
\begin{center}
\scalebox{0.8}{
\includegraphics{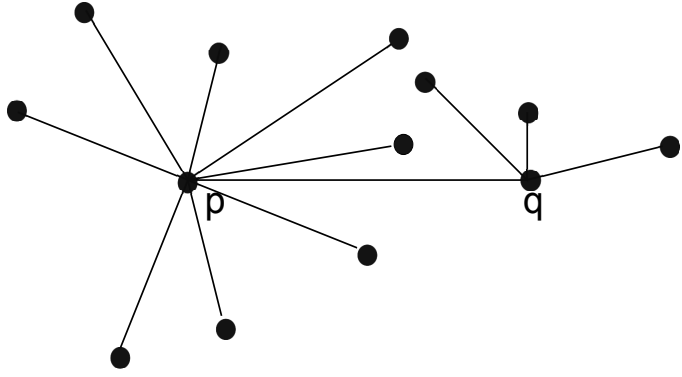}
} \caption{A $pq$-brush.}
\label{fig15p9}
\end{center}
\end{figure}

In this subsection we aim at identifying the vertices of $G(P)$ that represent brushes.
The identification uses distances in the graph $G(P)$, defined (as usual) as the length
of the shortest path between two vertices, and denoted by $d_{G(P)}(S,T)$. Note
that by the structure of $G(P)$, it is clear that for any pair of vertices $S,T \in G(P)$,
we have
\[
d_{G(P)}(S,T) \geq \frac{1}{2} |\Delta(S,T)|,
\]
where $\Delta(S,T)$ is the symmetric difference between the edge sets of the graphs $S$ and $T$.

\begin{theorem}\label{Thm:Brush}
An SST $T \in G(P)$ is a $pq$-brush if and only if it is not a star and
\begin{equation}\label{Eq:Brush1}
d_{G(P)} (T,S(p)) + d_{G(P)} (T,S(q)) = n-2.
\end{equation}
\end{theorem}

For sake of convenience, we divide the theorem into two propositions.
\begin{proposition}\label{Prop:Brush1}
Assume that $T \in G(P)$ satisfies
\[
d_{G(P)} (T,S(p)) + d_{G(P)} (T,S(q)) = n-2.
\]
Then $T$ is a $pq$-brush, or $T=S(p)$, or $T=S(q)$.
\end{proposition}

\begin{proof}
First, we note that \[d_{G(P)} (S(p),S(q)) \geq \frac12 |\Delta(S(p),S(q))| = n-2,\] and thus,
by the triangle inequality, \[d_{G(P)} (T,S(p)) + d_{G(P)} (T,S(q)) \geq n-2\] for any
$T \in G(P)$.

\medskip

Assume that $T$ satisfies~(\ref{Eq:Brush1}). Let $k,\ell$ be the numbers of edges of $T$ that emanate from $q,p$
(respectively), and let $r$ be the number of edges of $T$ that use neither $p$ nor $q$.
We consider two cases:
\begin{enumerate}
\item \textbf{$[p,q] \not \in T$.} In this case, we have
\begin{align*}
&k+\ell+r=n-1, \qquad  d_{G(P)} (T,S(p)) \geq \frac12 |\Delta(T,S(p))| = k+r, \qquad \mbox{and}
\qquad \\
&d_{G(P)}(T,S(q)) \geq \frac12 |\Delta(T,S(q))| = \ell+r.
\end{align*}
Hence,
\[d_{G(P)} (T,S(p))+ d_{G(P)} (T,S(q)) \geq k+\ell+2r=(n-1)+r>n-2.\]

\item \textbf{$[p,q] \in T$.} In this case,
\begin{align*}
&k+\ell+r=n, \qquad  d_{G(P)} (T,S(p)) \geq \frac12 |\Delta(T,S(p))| = k+r-1, \qquad \mbox{and}
\qquad \\
&D_{G(P)}(T,S(q))| \geq \frac12 |\Delta(T,S(q)) = \ell+r-1.
\end{align*}
Hence,
\[d_{G(P)} (T,S(p))+ d_{G(P)} (T,S(q)) \geq k+\ell+2r-2=n+r-2 \geq n-2,\]
and equality can hold
only if $r=0$, which means that all edges of $T$ emanate either from $p$ or from $q$ and
$[p,q] \in T$, i.e., $T$ is a $pq$-brush, or $T=S(p)$ or $T=S(q)$.
\end{enumerate}
\end{proof}

\begin{proposition}\label{Prop:Brush2}
Let $T \in G(P)$ be a $pq$-brush. Then
\[
d_{G(P)} (T,S(p)) + d_{G(P)} (T,S(q)) = n-2.
\]
\end{proposition}

In order to prove Proposition~\ref{Prop:Brush2}, we need a lemma.
\begin{definition}
Let $T$ be a $pq$-brush (or $T=S(p)$ or $T=S(q)$). We say that $T$ is \emph{of type $(k,\ell)$}
if $\mathrm{val}(T,p)=k+1$ and $\mathrm{val}(T,q)=\ell+1$ (i.e., the numbers of edges of $T$
that emanate from $p,q$ are $k+1,\ell+1$, respectively).
\end{definition}

\begin{lemma}\label{Lemma:Brush1}
If $T$ is a $pq$-brush (or a star) of type $(k,\ell)$, $k<n-2$, then it is adjacent in $G(P)$ to some $pq$-brush
(or star) $T'$ of type $(k+1,\ell-1)$. (By symmetry, if $\ell<n-2$ then $T$ is adjacent in $G(P)$ to
some $pq$-brush (or star) $T''$ of type $(k-1,\ell+1)$.)
\end{lemma}

\begin{proof}[Proof of the Lemma]
Assume w.l.o.g. that $[p,q]$ is placed horizontally, and consider the half-plane above it.
Let $[p,x]$ be an edge such that the angle $\alpha$ between $[p,q]$ and $[p,x]$ is minimal amongst
all edges of $T$ that emanate from $p$ (see Figure~\ref{fig16p10}). Let
$T' = (T \setminus \{[p,x]\}) \cup \{[q,x]\}$. Due to the minimality of the angle $\alpha$, $[q,x]$
does not cross any of the edges of $T$ that emanate from $p$. Thus, $T'$ is a $pq$-brush of
type $(k+1,\ell-1)$ (or $T'=S(q)$, if $\ell=1$), and $[T,T'] \in G(P)$ since $|\Delta(T,T')|=2$.
\end{proof}

\begin{proof}[Proof of Proposition~\ref{Prop:Brush2}]
Let $T \in G(P)$ be a $pq$-brush. Assume w.l.o.g. that $T$ is of type $(k,\ell)$. Repeated use
of Lemma~\ref{Lemma:Brush1} enables us to construct a path $\langle T_0,T_1,\ldots,T_{n-2} \rangle$
in $G(P)$ such that $T_0=S(q)$, $T_{k}=T$, and $T_{n-2}=S(p)$. This implies $d_{G(P)} (T,S(p)) \leq \ell$
and $d_{G(P)} (T,S(q)) \leq k$, hence, $d_{G(P)} (T,S(p)) + d_{G(P)} (T,S(q)) \leq \ell+k=n-2$. Since
we have shown above that $d_{G(P)} (T,S(p)) + d_{G(P)} (T,S(q)) \geq n-2$ for any $T$, this completes
the proof.
\end{proof}

\begin{figure}[tb]
\begin{center}
\scalebox{0.8}{
\includegraphics{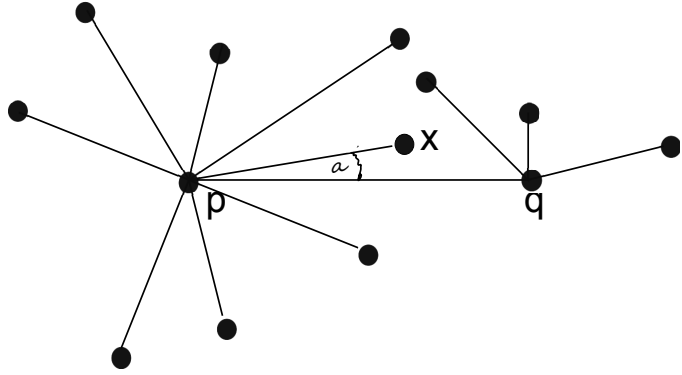}
} \caption{An illustration to the proof of Lemma~\ref{Lemma:Brush1}.}
\label{fig16p10}
\end{center}
\end{figure}

\section{Identification of the Geometric Structure of $K(P)$}
\label{sec:geom-structure}

In this section we achieve a complete reconstruction of the geometric structure of $K(P)$,
based on the identification of stars and brushes presented in Section~\ref{sec:stars-and-brushes}.
Most of the effort is devoted to obtaining a complete identification of the brushes, in the sense
that given a $pq$-brush $T$ and a vertex $x \neq p,q$, we determine whether $[x,p] \in T$ or
$[x,q] \in T$. This step is presented in Section~\ref{sec:sub:brushes-further}. The finalization
of the proof of Theorem~\ref{Thm:Main}, presented in Section~\ref{sec:sub:completion}, is easy.

\subsection{Further Information on Brushes in $G(P)$}
\label{sec:sub:brushes-further}

So far, we know which vertices of $G(P)$ are brushes. Furthermore, if we identify
the points of $P$ with the stars in $G(P)$ (an identification that is determined only up to an
automorphism of $K(P)$), we can say for each brush $T$, what are the vertices $p,q$ that
are its ``centers'', and how many edges of $T$ emanate from each of the central vertices $p,q$.

Our goal now is to gain full information on the brushes. Namely, for a $pq$-brush $T$
and a vertex $x \neq p,q$, we would like to determine whether $[p,x] \in T$ or
$[q,x] \in T$.

As an intermediate step, we would like to determine, for given $x,y \neq p,q$, whether both
$x$ and $y$ are connected in $T$ to the same vertex (either $p$ or $q$), or one of
them is connected to $p$ and the other to $q$.

\begin{proposition}\label{Prop:Brush3}
Let $T$ be a $pq$-brush, and let $x,y \in P$ be different from each other and from $p$ and $q$. 
The leaf edges of
$T$ whose endpoints are $x$ and $y$ emanate from the same internal vertex of $T$ if
and only if for any $xy$-brush $S$, we have $d_{G(P)} (T,S) \geq n-2$.
\end{proposition}

One direction of the proposition is immediate. If the leaf edges emanate from the same
vertex, e.g., $[p,x],[p,y]$, then at most one (and actually, exactly one) of these edges can belong
to an $xy$-brush
(as $([x,y],[y,p],[p,x])$ form a cycle). Since every edge of an $xy$-brush $S$ emanates from
either $x$ or $y$, we have $\frac{1}{2}|\Delta(S,T)| \geq n-2$ (as these trees have exactly
one edge in common). Hence,
\[
d_{G(P)} (T,S) \geq \frac{1}{2}|\Delta(S,T)| = n-2
\]
for any $xy$-brush $S$.

On the other hand, if the leaf edges emanate from different vertices, e.g., $[p,x],[q,y]$,
it is possible that an $xy$-brush $S$ include both these edges, and then
$\frac12 |\Delta(S,T)|=n-3$. We will construct an $xy$-brush that satisfies this condition,
and furthermore, satisfies the stronger condition $d_{G(P)}(S,T)=n-3$. Before we show this
construction, we need a few preparations.
\begin{definition}
Let $G$ be a geometric graph, and let $O$ be a point in the plane. We say that $O$ \emph{sees} a
point $P$ if the open segment $(O,P)$ does not meet any edge or vertex of $G$. We say that
$O$ sees an edge $e \in E(G)$ if it sees every point $X \in e$, including the endpoints.
\end{definition}

\begin{lemma}\label{Lemma:Sees}
Let $G=(V,E)$ be a crossing-free geometric graph, with no isolated vertices.
Suppose $V$ is a disjoint union $V=V_0 \cup W$, where $|V_0|=2$ (say, $V_0=\{p,q\}$), and
each edge of $G$ connects a vertex of $V_0$ with a vertex of $W$. Suppose $O \in \mathbb{R}^2
\setminus \bigcup \{\mathrm{aff}(e):e \in E(G)\}$. Then $O$ sees some vertex $w \in W$.
\end{lemma}

We note that a similar lemma was proved in~\cite{Hernando}. The assumption on $O$
in~\cite{Hernando} is $O \not \in \mathrm{conv}(V(G))$, and the assertion is the same as
in our lemma.

\begin{proof}
Draw a ray $R$ that emanates from $O$, crosses some edge of $G$, and does not meet any vertex
of $G$. (It is clear that such rays exist.) Denote by $C$ the first crossing point of $R$
with an edge of $G$. Then $C$ is an interior point of an edge, say $[p,w]$, of $G$, and $O$ sees
$C$. Now rotate $R$ around $O$ towards $w$, until it hits $w$. If the triangle $\triangle OCw$ does not
contain any vertex of $G$, except $w$, then $O$ sees $w$. Otherwise, there is a first position
$R'$ of the rotated ray that meets $V$. Let $v$ be the point of $R' \cap V$ closest to $O$. Then
$O$ sees $v$. If $v \in W$, we are done. Assume, therefore, that $v \in V_0$. Clearly, $v \neq p$
since $0< \angle vOp < \angle wOp < \pi$. Hence, $v=q$.

Among the edges that emanate from $q$, let $[q,w']$ be the edge such that $\angle Oqw'$
is minimal. As before, rotate $R'$ around $O$ towards $w'$, until it hits $w'$. If the triangle
$\triangle Oqw'$ does not contain any vertex of $G$, except $q$ and $w'$, then $O$ sees $w'$. Otherwise,
there is a first position $R''$ of the rotated ray that meets $V$. Let $v'$ be the point of
$R'' \cap V$ closest to $O$. Then $O$ sees $v'$. Now, we observe that $v' \not \in V_0$.
Indeed, $v' \neq q$ since $0< \angle v'Oq < \angle qOw' < \pi$, and $v' \neq p$ since
$0 < \angle v'Op = \angle v'Oq + \angle qOp < \angle w'Oq + \angle wOp < 2\pi$. Therefore,
$v' \in W$, which completes the proof.
\end{proof}

Now we are ready to prove Proposition~\ref{Prop:Brush3}.
\begin{proof}[Proof of Proposition~\ref{Prop:Brush3}]
We already proved above that if the two leaf edges of $T$ whose endpoints are $x,y$
emanate from the same vertex, then for any $xy$-brush $S$, $d_{G(P)} (T,S) \geq n-2$.
Assume now that these leaf edges emanate from different vertices. W.l.o.g., these
edges are $[p,x],[q,y]$. We consider two cases, according to the placement
of $p,q,x,y$ in the plane. In each case, we show that we can pass from $T$ to
a suitable $xy$-brush $S$ in $n-3$ steps, where in each step we remove
one edge and add another edge, while maintaining the simplicity. This will show that
$d_{G(P)}(T,S) = n-3$, and thus complete the proof of the proposition. Note that in
all the steps of the path connecting $T$ to $S$, the edges $[p,x],[q,y]$ remain
untouched.

\medskip

\noindent \textbf{Case 1:  $x,y$ are on the same side of $\ell(p,q)$.}

In this case, at least one of the edges $[p,x],[q,y]$ is included in a line that
supports the set $\{p,x,q,y\}$ (which means that all points in the set are on the
same side of the line). We assume w.l.o.g. that $[p,x]$ has this property.

\begin{figure}[tb]
\begin{center}
\scalebox{0.8}{
\includegraphics{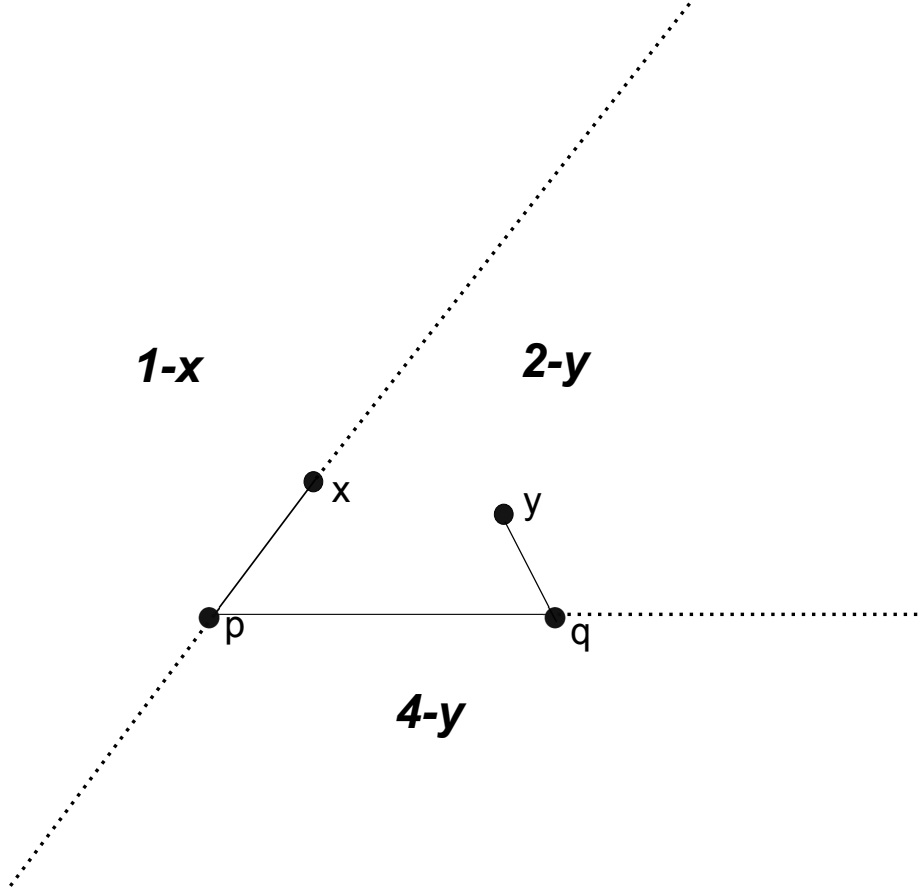}
} \caption{An illustration to the proof of Proposition~\ref{Prop:Brush3}: Case 1.
The regions are numbered by the order of their consideration,
where the missing number 3 corresponds to Phase~3 in which $[p,q]$ is replaced by
$[x,y]$. The notation $(i) - (v)$ where $i \in \{1,2,4\}$ and $v \in \{x,y\}$ means that
we are going to connect all points in Region~$i$ to $v$.}
\label{fig18p12new}
\end{center}
\end{figure}

The passage from $T$ to an appropriate $S$ is performed by a 4-phase procedure,
illustrated in Figure~\ref{fig18p12new}. In each phase (except for phase~3 that will be
described below), we consider one of the regions of the plane denoted in the figure:
$1,2,4$, and deal with all points of $P$ that belong to that region.
\begin{enumerate}
\item \textbf{Region 1 (Reg1).} This region is the open half-plane to the left of
the line $\ell(p,x)$. Assume that $|P \cap Reg1|=k_1$. We are going to perform $k_1$
steps: in each step, we take one of these points, remove the edge that connects it to
either $p$ or $q$, and add an edge that connects it to $x$. Of course, we must maintain
the simplicity during all steps, and this is achieved using Lemma~\ref{Lemma:Sees}.

\medskip

Let $G_0$ be the geometric graph whose edges are all edges of $T$ of the form $[p,w]$ or
$[q,w]$, where $w$ lies in $Reg1$. The graph $G_0$ and the point $O=x$ satisfy the
assumptions of Lemma~\ref{Lemma:Sees}, and thus, by the Lemma, $x$ sees one of its vertices
$w \in Reg1$, say $w_1$. Assume, for example, that $[q,w_1] \in E(G_0)$. Define
$T_1 = T \setminus \{[q,w_1]\} \cup \{[x,w_1]\}$. Since $x$ sees $w_1$,
the edge $[x,w_1]$ does not cross any other edge of $T$. Thus, $T_1$ is an SST and
$|T \triangle T_1|=2$, which implies that $[T,T_1] \in G(P)$.

Now, we repeat the first step with the SST $T_1$ in place of $T$. That is, we define $G_1$
whose edges are all edges of $T$ of the form $[p,w]$ or $[q,w]$ where $w$ lies in $Reg1$,
except for $[q,w_1]$. As before, we apply Lemma~\ref{Lemma:Sees} with $G_1$ and $O=x$ and
obtain a vertex $w_2$ that is seen from $x$. Then, we define $T_2$ by removing from $T_1$
the edge that connects $w_2$ to either $p$ or $q$ and adding the edge $[x,w_2]$. Note that
the edge $[x,w_1]$ that was not included in $G_2$ cannot cross $[x,w_2]$, as they both
emanate from $x$.

By continuing in the same fashion, we obtain a sequence $T_0,T_1,\ldots,T_{k_1}$ such that
$T_0=T$, $[T_i,T_{i+1}] \in G(P)$ for all $i$, and in $T_{k_1}$, all points in $P \cap Reg1$
are connected to $x$.

It should be noted that the parts of $T_i$ that are not included in the auxiliary graph $G_i$,
i.e., the edge $[p,q]$ and the edges $[p,w],[q,w]$, $w \in \mathbb{R}^2 \setminus Reg1$, are
all disjoint from the convex set $Reg1$, and thus cannot cross the new edge $[x,w_{i+1}]$ (as
$x \in \mathrm{bdry}(Reg1)$).

\item \textbf{Region 2 (Reg2).} This region contains all points that lie above $\ell(p,q)$ and on the
right side of $\ell(p,x)$. Assume that $|P \cap Reg2|=k_2$.
We start with $T_{k_1}$ and perform $k_2$ steps: in each step, we consider one of these
points, remove the edge that connects it to either $p$ or $q$, and add an edge that connects it
to $y$. As before, the simplicity is maintained during all steps, by using Lemma~\ref{Lemma:Sees}.

Let $G_{k_1}$ be the geometric graph whose edges are all edges of $T_{k_1}$ of the form $[p,w]$ or
$[q,w]$, where $w$ lies in $Reg2$. The graph $G_{k_1}$ and the point $O=y$ satisfy the
assumptions of Lemma~\ref{Lemma:Sees}, and thus, by the Lemma, $y$ sees one of the vertices
$w \in Reg2$, call it $w_{k_1+1}$. Without loss of generality, $[p,w_{k_1+1}] \in G_{k_1}$. Define
$T_{k_1+1} = T_{k_1} \setminus \{[p,w_{k_1+1}]\} \cup \{[y,w_{k_1+1}]\}$. Since $y$ sees $w_{k_1+1}$,
the edge $[y,w_{k_1+1}]$ does not cross any other edge of $T_{k_1}$. Thus, $[T_{k_1},T_{k_1+1}] \in G(P)$.

By continuing in the same fashion, we obtain a sequence $T_{k_1+1},T_{k_1+2},\ldots,T_{k_1+k_2}$ such that
$[T_i,T_{i+1}] \in G(P)$ for all $i$, and in $T_{k_1+k_2}$, all points in $P \cap Reg1$ are connected
to $x$ and all points in $P \cap Reg2$ are connected to $y$.

\item \textbf{Phase 3.} In this phase, we add the edge $[x,y]$ and remove the edge $[p,q]$ (that otherwise
closes a cycle $([x,y],[y,q],[q,p],[p,x])$). Formally, we define $T_{k_1+k_2+1} =
T_{k_1+k_2} \setminus \{[p,q]\} \cup \{[x,y]\}$. Note that the edge $[x,y]$ does not cross any edge of
$T_{k_1+k_2}$, as in $T_{k_1+k_2}$, all points of $P \cap Reg2$ are connected to $y$.

\item \textbf{Region 4 (Reg4).} This region contains all points that lie below $\ell(p,q)$ and on the
right of $\ell(p,x)$. Assume $|P \cap Reg4|=k_3$. As all points of $P$ except for $p,q,x,y$
belong to one of the regions: $Reg1,Reg2,Reg4$, we have $k_1+k_2+k_3=n-4$. We construct
a sequence of SSTs $T_{k_1+k_2+2},\ldots,T_{k_1+k_2+k_3+1}$ such that in $T_{k_1+k_2+k_3+1}$,
all points in $P \cap Reg1$ are connected to $x$ and all points in $P \cap (Reg2 \cup Reg4)$ are connected
to $y$. Hence, $T_{k_1+k_2+k_3+1}=T_{n-3}$ is an $xy$-brush that satisfies $d_{G(P)} (T,T_{n-3}) = n-3$,
as desired.

Let $G_{k_1+k_2+1}$ be the geometric graph whose edges are all edges of $T_{k_1+k_2+1}$ of the
form $[p,w]$ or $[q,w]$, where $w$ lies in $Reg4$. The graph $G_{k_1+k_2+1}$ and the point $O=y$
satisfy the assumptions of Lemma~\ref{Lemma:Sees}, and thus, by the Lemma, $y$ sees one of the vertices
$w \in Reg4$, say $w_{k_1+k_2+1}$. Without loss of generality, $[p,w_{k_1+k_2+1}] \in G_{k_1+k_2+1}$. Define
$T_{k_1+k_2+2} = T_{k_1+k_2+1} \setminus \{[p,w_{k_1+k_2+1}]\} \cup \{[y,w_{k_1+k_2+1}]\}$. Since $y$ sees
$w_{k_1+k_2+1}$, the edge $[y,w_{k_1+k_2+1}]$ does not cross any other edge of $T_{k_1}$. (It should be noted that the fact that
$y$ lies outside $Reg4$ does not disturb us, as all points in Region~2 (that is the region $y$ sees $Reg4$
through) are already connected to $y$, and $P$ is in general position.) Thus, $[T_{k_1+k_2+1},T_{k_1+k_2+2}] \in G(P)$.

By continuing in the same fashion, we obtain a sequence $T_{k_1+k_2+2},T_{k_1+k_2+3},\ldots,T_{k_1+k_2+k_3+1}$ such that
$[T_i,T_{i+1}] \in G(P)$ and $T_{k_1+k_2+k_3+1}=T_{n-3}$ is the desired $xy$-brush.
\end{enumerate}

\medskip

\noindent \textbf{Case 2:  $x,y$ are on different sides of $\ell(p,q)$.}

This case is treated in a fashion similar to Case~1. We divide all points of $P \setminus \{p,q\}$ into
two regions, where Region 1 (Reg1) consists of the points above $\ell(p,q)$ and Region 2 (Reg2)
consists of the points below $\ell(p,q)$ (see Figure~\ref{fig21p14new}). In the first phase, we
consider the points of $Reg1$ (excluding $x$), disconnect them from $p$ or $q$ and connect them to $x$ instead.
The procedure is identical to the procedure of the first phase of Case~1. In the second phase, we
consider the points of $Reg2$ (excluding $y$), disconnect them from $p$ or $q$ and connect them to $y$ instead.
The procedure is, again, similar. Finally, in the third phase we remove the edge $[p,q]$ and
insert the edge $[x,y]$ instead. (As at this stage, all points in $P \setminus \{p,q,x,y\}$
are connected to either $x$ or $y$, this step does not create crossings.) As a result, we obtain
a sequence $T=T_0,T_1,T_2,\ldots,T_{n-3}$, such that $[T_i,T_{i+1}] \in G(P)$ for all $i$ and
$T_{n-3}$ is an $xy$-brush, as desired.

\begin{figure}[tb]
\begin{center}
\scalebox{0.8}{
\includegraphics{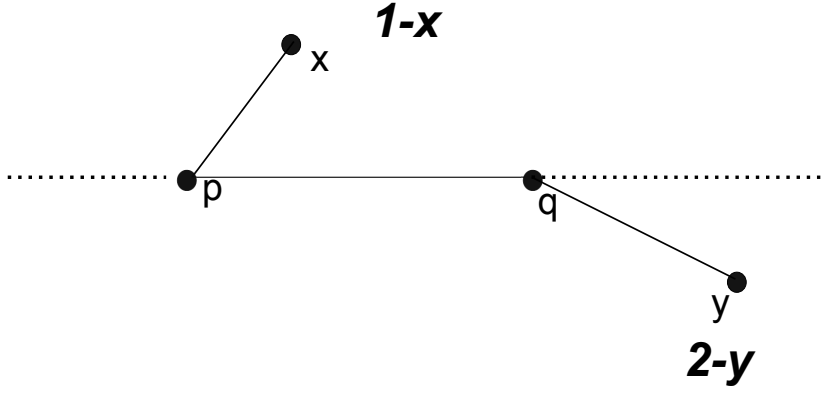}
} \caption{An illustration to the proof of Proposition~\ref{Prop:Brush3}: Case 2.
The notation $(i) - (v)$ where $i \in \{1,2\}$ and $v \in \{x,y\}$ means that
the points in Region~$i$ are connected to $v$.}
\label{fig21p14new}
\end{center}
\end{figure}

As Cases~1,2 include all possible placements of $x,y,p,q$, the proof is complete.
\end{proof}

Now we are ready to identify every brush completely.
\begin{corollary}\label{Cor:Brush}
Let $T \in V(G(P))$ be a $pq$-brush and let $x \in P$, $x \neq p,q$. Given $G(P)$, we
can determine whether $[p,x] \in T$ or $[q,x] \in T$.
\end{corollary}

\begin{proof}
It follows from the proof of Proposition~\ref{Prop:Brush2} that $T$ belongs to a path
\[
\langle S(p)=T_0, T_1, T_2,\ldots,T_{n-3},T_{n-2}=S(q) \rangle
\]
in $G(P)$ such that in $T_i$,
$\deg(p)=n-1-i$ and $\deg(q)=i+1$. Consider $T_1$. Since it has only one vertex $x_1 \neq p,q$
that is connected to $q$, we can use Proposition~\ref{Prop:Brush3} to determine it. (Here we
use the assumption that $n \geq 5$.) We can then
move to $T_2$ and use Proposition~\ref{Prop:Brush3} again to determine the additional vertex $x_2$
connected to $q$ in $T_2$. (Note that $[x_1,q] \in E(T_2)$, and thus, $x_2$ is identified as the
unique vertex $x$ such that the leaf edge of $T_2$ that emanates from it has the same second endpoint
as the leaf edge that emanates from $x_1$.) We can continue in the same fashion and get a complete
identification of $T_1,T_2,\ldots,T_{n-3}$, including $T$.
\end{proof}

\subsection{Completing the Proof of Theorem~\ref{Thm:Main}}
\label{sec:sub:completion}

Our last step toward the identification of the geometric structure of $K(P)$ is the following
easy proposition.
\begin{proposition}\label{Prop:Brush4}
Let $p,q,x,y$ be four different points in $P$. The segments $[p,x]$ and $[q,y]$ do not cross if and only if
there exists a $pq$-brush that includes the edges $[p,x]$ and $[q,y]$.
\end{proposition}

\begin{proof}
It is clear that if $[p,x]$ and $[q,y]$ cross then no $pq$-brush can contain both edges
$[p,x]$ and $[q,y]$, as a brush is a \emph{simple} tree. If $[p,x]$ and $[q,y]$ do not cross, then
they are strictly separated by some line $\ell$. In such a case, we can define a $pq$-brush in which all vertices
that lie on the same side of $\ell$ as $p$ are connected to $p$, and all other vertices are connected
to $q$ (see Figure~\ref{fig22p14}). This $pq$-brush includes both $[p,x]$ and $[q,y]$.

\begin{figure}[tb]
\begin{center}
\scalebox{0.8}{
\includegraphics{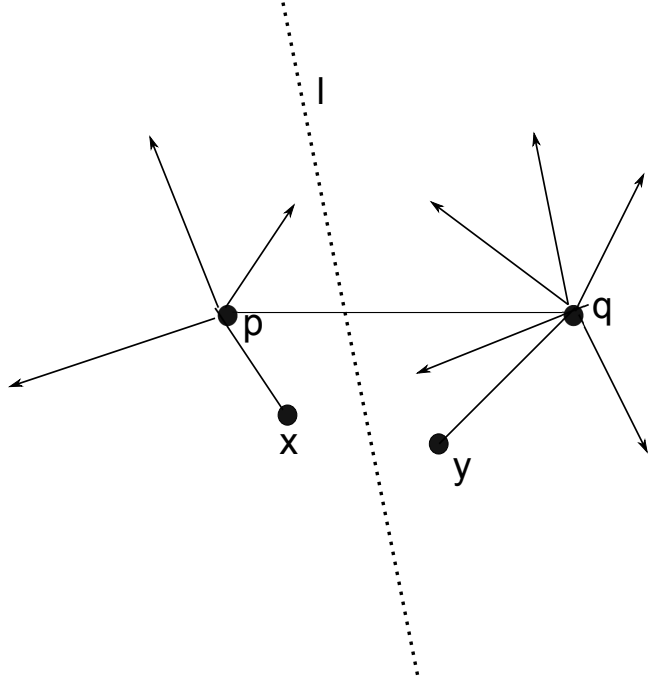}
} \caption{An illustration to the proof of Proposition~\ref{Prop:Brush4}.}
\label{fig22p14}
\end{center}
\end{figure}

\end{proof}

Now we are ready to prove our main theorem.

\begin{proof}[Proof of the Main Theorem]
Consider the geometric tree graph $G(P)$. The vertices $x,y,p,q$ are identified with the stars
$S(x),S(y),S(p),S(q) \in G(P)$. By Theorem~\ref{Thm:Brush}, we can identify all $pq$-brushes
in $G(P)$. By Corollary~\ref{Cor:Brush}, we can check for each of them whether it includes
both $[p,x]$ and $[q,y]$ or not. By Proposition~\ref{Prop:Brush4}, if none of the $pq$-brushes
contains  both $[p,x]$ and $[q,y]$, then these segments cross, and otherwise, they do not
cross. This completes the proof of the theorem.
\end{proof}

\section{The Automorphism Group of the Tree Graph of $K_n$}
\label{sec:general}

In this section we consider the abstract (i.e., non-geometric) graph $K_n$. Recall that,
as defined in the introduction, the vertices of the tree graph $G(K_n)$ are all the spanning 
trees of $K_n$, and two spanning trees are adjacent if they differ in exactly two edges.
We prove Theorem~\ref{Thm:Main-General}, stating that the automorphism group of $G(K_n)$
is isomorphic to $\mathrm{Aut}(K_n)=S_n$.

It turns out that the theorem can be proved by roughly the same methodology as the proof of
Theorem~\ref{Thm:Main}, as shown below. Altogether, the proof in the abstract setting 
turns out considerably simpler than its geometric counterpart.

\medskip

\noindent \textbf{Identification of stars in $G(K_n)$.} Denote $G=G(K_n)$, and let 
$V(K_n)=\{v_1,\ldots,v_n\}$. As in the geometric case, our first step is identification of
the vertices of $G$ that represent stars. Unlike the geometric case, here the identification 
is immediate.
\begin{claim}
Let $T \in V(G)$. Then $T$ represents a star if and only if for any $T' \in V(G)$, 
$d_G(T,T') \leq n-2$. 
\end{claim}

\begin{proof}
We observe that in the abstract case, $d_G(S_1,S_2) = \frac12 |\Delta(S_1,S_2)|$
for any $S_1,S_2 \in V(G)$. (In the geometric case, we could only say that
$d_{G(P)}(S_1,S_2) \geq \frac12 |\Delta(S_1,S_2)|$.) 

Assume that $T$ is a star. Since any spanning tree $T'$ of $K_n$ shares at least one edge with
$T$, we have $d_G(T,T') = \frac12 |\Delta(T,T')| \leq n-2$. 

On the other hand, if $T$ is not a star then it is easy to see that the graph 
$T^c = K_n \setminus T$ is connected, and thus, there exists $T' \in V(G)$ that does not share
an edge with $T$. Hence, $d_G(T,T') = \frac12 |\Delta(T,T')| = n-1$.
\end{proof}

We note that since the distance in $G$ between any pair of stars in $n-2$, it follows that 
the quantity $\max_{\{T' \in V(G): T'\neq T\}} d_G(T,T')$ equals $n-2$ if $T$ represents a star
and $n-1$ otherwise. Consequently, the set of vertices that represent stars is exactly 
the {\it center} of the graph $G$.  

These vertices can be identified with the vertices of $K_n$ in an arbitrary way (as any automorphism of
$K_n$ clearly induces an automorphism of $G$). So, we call the $n$ vertices in $G$
that represent stars $S(v_1),\ldots,S(v_n)$ (in some arbitrary order).

\medskip

\noindent \textbf{Valences of vertices in $G(K_n)$.} The next simple step is identifying, for any 
$T \in V(G)$ and any vertex $v$, what is the valence of $v$ in $T$. Note that we were not able to 
obtain such an identification in the geometric setting.
\begin{claim}
Let $T \in V(G)$ and $v \in V(K_n)$. The valence $\delta_T(v)$ of $v$ in $T$ is $n-1-d_G(T,S(v))$.
\end{claim}

\begin{proof}
Since all edges in $S(v)$ emanate from $v$, it is clear that
$\frac{1}{2} |\Delta(T,S(v))| = n-1-\delta_T(v)$. As $d_G(T,S(v)) = \frac12 |\Delta(T,S(v))|$, the
assertion follows.
\end{proof}

\medskip

\noindent \textbf{Max-cliques in $G(K_n)$.} Our next step is examination of max-cliques
in $G$. As in the geometric case, we would like to determine whether a given max-clique
is a $U$-clique or an $I$-clique.
\begin{proposition}\label{Lemma:UI-cliques}
Given a max-clique $\mathcal{C}$ of $G$, we can determine whether it is a $U$-clique
or an $I$-clique.
\end{proposition}

\begin{proof}
As the discussion in Section~\ref{sec:sub:max-cliques-basics} is purely combinatorial, 
it applies without change to the abstract setting. In particular, all vertices in a 
$U$-clique $U(S,T)$ are obtained from $S \cup T$ by removing an edge from its unique 
cycle, and all vertices in an $I$-clique $I(S,T)$ are obtained from the two-component 
forest $S \cap T$ by adding an edge that connects its two components. As there are 
no geometric restrictions in our case, it follows that $|U(S,T)|$ is equal to the 
size of the unique cycle in $S \cup T$ (and, in particular, is between $3$ and $n$), 
and $|I(S,T)|=k(n-k)$, where $k$ is the number of vertices in one of the connected 
components of $S \cap T$. Hence, determination whether $\mathcal{C}$ is a $U$-clique 
or an $I$-clique is non-trivial only if $|\mathcal{C}|=n-1$.

A $U$-clique $U(S,T)$ is of size $n-1$ if the unique cycle $C$ of $S \cup T$ is of size $n-1$,
which means that $S \cup T$ consists of $C$ plus a single additional edge. Each element of
$U(S,T)$ is obtained from $S \cup T$ by removing one edge from $C$. Assume w.l.o.g. that
$C= \langle v_1,v_2,\ldots,v_{n-1},v_1 \rangle$, and the additional edge is $[v_n,v_1]$. It is clear that
$v_1$ is never a leaf in a tree of $U(S,T)$, $v_n$ is a leaf in all $n-1$ trees of $U(S,T)$,
and each of the vertices $v_2,\ldots,v_{n-1}$ is a leaf in exactly two trees of $U(S,T)$.
In addition, $U(S,T)$ has two trees that are paths. (These are the trees obtained by removing
$[v_1,v_2]$ and $[v_{n-1},v_1]$.) These two trees can be recognized by checking that their
sequence of valences is $1,2,2,\ldots,2,1$.

An $I$-clique $I(S,T)$ is of size $n-1$ if in the two-component forest $S \cap T$, one
component consists of a single vertex $x$. Assume, in addition, that $I(S,T)$ has two elements
that are paths. (Otherwise, we can determine that $I(S,T)$ is an $I$-clique by the previous
paragraph.) This is possible only if the second component of $S \cap T$ is a path $P$. In such a
case, each endpoint of $P$ is a leaf in $n-2$ (of the $n-1$) trees of $I(S,T)$.
As in $U(S,T)$, all vertices except one are leaves in at most two trees of $U(S,T)$,
this property allows to determine that $I(S,T)$ is indeed an $I$-clique.
\end{proof}

\medskip

\noindent \textbf{The automorphism group of $G(K_n)$.} Our last step is to show that the
information on $G$ obtained so far is sufficient for determining uniquely the spanning tree 
represented by each vertex of $G$. Namely, given $T \in V(G)$ and two vertices $p,q \in V(K_n)$, 
we would like to determine whether $[p,q] \in E(T)$ or not. If this is possible, it implies that
$\mathrm{Aut}(G) \cong \mathrm{Aut}(K_n) \cong S_n$, since our determination is unique up to the 
arbitrary identification of the vertices of $G$ that represent stars with the vertices of $K_n$.
Hence, this will complete the proof of Theorem~\ref{Thm:Main-General}. 

First, we consider the case when neither $p$ nor $q$ is a leaf in $T$.
\begin{claim}
Let $T \in V(G)$ and suppose $p,q \in V(K_n)$, $p \neq q$, $\delta_T(p),\delta_T(q) \geq 2$. Then
$[p,q] \in E(T)$ if and only if $T$ has a neighbor $T'$ in $G$ in which the valences of both
$p$ and $q$ are smaller by $1$ than in $T$.
\end{claim}

\begin{proof}
If $[p,q] \not \in E(T)$ then no removal of an edge from $E(T)$ can reduce the valences of
both $p$ and $q$, and thus, $T'$ as described in the claim does not exist. On the
other hand, if $[p,q] \in E(T)$ then the graph $\tilde{T} = T \setminus \{[p,q]\}$ is a two-component forest
in which both components are of size $\geq 2$. Hence, there exists an edge $[p',q']$ that connects
the two components of $\tilde{T}$ and uses neither $p$ nor $q$. The tree $T' = T \setminus \{[p,q]\}
\cup \{[p',q']\}$ is a neighbor of $T$ as described in the claim.
\end{proof}

\noindent Now we can assume w.l.o.g. that $p$ is a leaf in $T$. We perform a four-step procedure:
\begin{enumerate}
\item Find a leaf $p'$ of $T$ such that $d_T(p,p') >2$.

\item Find a neighbor $T'$ of $T$ in $G(K_n)$ such that $\delta_{T'}(p),\delta_{T'}(p') \geq 2$.

\item Consider the $U$-clique $U(T,T')$, and find a tree $S \in U(T,T')$ such that $\delta_S(p)=1$
and $\delta_S(p')=2$.

\item We find a vertex $p''$ such that $\delta_S(p'') = \delta_T(p'')-1$. We claim that if $p''=q$
then $[p,q] \in E(T)$, and otherwise, $[p,q] \not \in E(T)$.
\end{enumerate}
We show below that the four steps can indeed be performed, and that they allow to determine
whether $[p,q] \in E(T)$ or not, as claimed.

\medskip

\noindent \textbf{Step 1.} First, we note that if $d_T(p,p')=2$ for all leaves $p'$ of $T$, then $T$ is a
star, and thus, $[p,q] \in E(T)$ for the unique $q$ whose valence in $T$ is greater than $1$ and
$[p,q] \not \in E(T)$ for any other $q$. Hence, we may assume that there exists a leaf $p'$ such
that $d_T(p,p')>2$, and we only have to detect it.

Consider the set of leaves of $T$ other than $p$: $A = \{p_i \in V(K_n): p_i \neq p, \delta_T(p_i)=1\}$.
(Note that we can recognize this set, as we are able to determine valences of vertices.)
We claim that $d_T(p,p') > 2$ if and only if there exists a neighbor $T'$ of $T$
in $G(K_n)$ such that $\delta_{T'}(p)=\delta_{T'}(p')=2$. This allows to detect the desired
$p'$ by going over the elements of $A$, and for each of them, going over the neighbors of $T$
in $G(K_n)$ and checking whether the claimed neighbor exists.

To see that the claim holds, note that a neighbor $T'$ of $T$ satisfies $\delta_{T'}(p)=\delta_{T'}(p')=2$,
if and only if it is of the form $T' = T \setminus \cup \{[p,p']\} \setminus \{e\}$, for an edge $e$
that belongs to the unique cycle $C$ of $T \cup \{[p,p']\}$ and uses neither $p$ nor $p'$. If
$d(p,p')=2$, then $C$ is of length $3$, and thus, it has no edges that use neither $p$
nor $p'$. Thus, no such neighbor $T'$ exists. If $d(p,p')>2$, then $C$ is of length $>3$, and thus,
it includes an edge $e$ that uses neither $p$ nor $p'$. The tree
$T' = T \setminus \{e\} \cup \{[p,p']\}$ is the desired neighbor of $T$.

\medskip

\noindent \textbf{Step 2.} This step is immediate, as the required neighbor $T'$ was already found
in Step~1.

\medskip

\noindent \textbf{Step 3.} The required neighbor $S$ is the tree obtained from $T \cup \{[p,p']\}$
by removing the unique edge of the cycle $C$ that uses $p$ but not $p'$ (call it $[p,p'']$).
The $U$-clique $U(T,T')$ can be recognized using Proposition~\ref{Lemma:UI-cliques}, since 
there exist only two max-cliques of
$G(K_n)$ that include both $T$ and $T'$ -- a $U$-clique and an $I$-clique -- and 
Proposition~\ref{Lemma:UI-cliques} allows us to determine, which of them is the $U$-clique. 
Then, $S$ can be recognized as the unique element of $U(T,T')$ in which the valences of 
$p,p'$ are $1$ and $2$, respectively.

\medskip

\noindent \textbf{Step 4.} It is clear that the unique vertex whose valence in $S$ is smaller by one
than its valence in $T$ is $p''$, as defined in Step~3. By the construction of $C$, $[p,p'']$ is
the unique edge of $E(T)$ that emanates from $p$, i.e., $p''$ is the unique neighbor of $p$ in $T$.
Hence, $[p,q] \in E(T)$ if and only if $q=p''$, as asserted. The vertex $p''$ is detected by
comparing the valences of the vertices in $S$ with their respective valences in $T$.

This completes the proof of Theorem~\ref{Thm:Main-General}.




\begin{thebibliography}{99}

\bibitem{AvisFukuda} D. Avis and K. Fukuda, Reverse Search for Enumeration,
Discrete Applied Mathematics \textbf{65(1)}, pp.~21--46, 1996.

\bibitem{Bondy} J. A. Bondy and R. L. Hemminger, Graph reconstruction -- a survey,
J. Graph Theory \textbf{l} (1977), pp.~227-–268.


\bibitem{Cummins} R. L. Cummins, Hamilton circuits in tree graphs, IEEE Trans.
Circuit Th., \textbf{13(1)} (1966), pp.~82--90.

\bibitem{ES35} P. Erd\H{o}s and G. Szekeres, A Combinatorial
Problem in Geometry, Compositio Math. \textbf{2}, pp.~463–-470, 1935.

\bibitem{Hernando1} M. C. Hernando, F. Hurtado, A. M$\mathrm{\acute{a}}$rquez,
M. Mora and M. Noy, Geometric Tree Graphs of Points in Convex Position, Discrete
Applied Mathematics \textbf{93(1)}, pp.~51--66, 1999.

\bibitem{Hernando} M. C. Hernando, Complejidad de Estructuras
Geom$\mathrm{\acute{e}}$tricas y Combinatorias, Ph.D. Thesis,
Universitat Polit$\mathrm{\acute{e}}$ctnica de Catalunya, 1999 (in
Spanish). Available online at:
http://www.tdx.cat/TDX-0402108-120036/

\bibitem{HH72} C. A. Holzmann and F. Harary, On the tree graph of a matroid, SIAM J. Appl.
Math. \textbf{22} (1972), pp.~187--193.


\bibitem{Kelly} P. J. Kelly, A congruence theorem for trees, Pacific J. Math. \textbf{7} (1957),
pp.~961–-968.

\bibitem{Liu88} G. Liu, On connectivities of tree graphs, J. Graph Theory \textbf{12} (1988),
pp.~453--459.

\bibitem{Rama} S. Ramachandran, Graph reconstruction -- some new developments, AKCE J. Graphs.
Combin., \textbf{1(1)} (2004), pp.~51--61.

\bibitem{Sedlacek} J. Sedl\'{a}\v{c}ek, The reconstruction of a connected graph from its
spanning trees, Mat. \v{C}asopis Sloven. Akad. Vied. \textbf{24} (1974), pp.~307--314.

\bibitem{Ulam} S. M. Ulam, A collection of mathematical problems, Wiley, New York, 1960.

\bibitem{Virginia-PhD} V. Urrutia-Galicia, Algunas Propiedades de Gr$\mathrm{\acute{a}}$ficas
Geom$\mathrm{\acute{e}}$tricas, Ph.D. Thesis, Universidad Auton$\mathrm{\acute{o}}$ma Metropolitana Unidad
Iztapalapa,  M$\mathrm{\acute{e}}$xico D.F., 2001 (in Spanish).



\end{thebibliography}
\end{document}